\documentclass[10pt]{article}
\newcommand\tab[1][1cm]{\hspace*{#1}}
\usepackage{enumitem}
\usepackage{amsthm}
\usepackage{amsmath,amssymb,wasysym}
\usepackage{graphicx}
\usepackage{breqn}
\usepackage{authblk}
\usepackage[notcite,notref]{}
\usepackage[%
  colorlinks=true,%
	linkcolor	=blue,%
	citecolor =blue,%
  urlcolor=blue%
]{hyperref}
\usepackage[left=3cm,right=3cm, top=3.3cm, bottom=3cm]{geometry}
\newtheorem{theo}{Theorem}[section]
\newtheorem{lemm}[theo]{Lemma}

\newtheorem{defi}[theo]{Definition}
\newtheorem{remark}[theo]{Remark}
\numberwithin{equation}{section}

\title{Global existence for an attraction-repulsion chemotaxis fluid model with logistic source}
\author[a]{Abelardo Duarte-Rodr\'{\i}guez}
\author[b]{Lucas C. F. Ferreira\thanks{Corresponding author.\\ \indent E-mail addresses: \href{mailto:abelneonmec@gmail.com}{abelneonmec@gmail.com} (A. Duarte-Rodr\'{\i}guez), \href{mailto:lcff@ime.unicamp.br}{lcff@ime.unicamp.br}  (L. C. F. Ferreira),\\ \href{mailto:jvillami@uis.edu.co}{jvillami@uis.edu.co}  (E. J. Villamizar-Roa).}}
\author[a]{\'Elder J. Villamizar-Roa}
\affil[a]{Universidad Industrial de Santander, Escuela de Matem\'{a}ticas, A.A. 678, Bucaramanga, Colombia.}
\affil[b]{Universidade Estadual de Campinas, Departamento de Matemática, CEP 13083-859, Campinas-SP, Brazil.}

\date{}
\begin{document}
\maketitle
\begin{abstract}
We consider an attraction-repulsion chemotaxis model coupled with the Navier-Stokes system. This model describes the interaction between a type of cells (e.g., bacteria), which proliferate following a logistic law, and two chemical signals produced by the cells themselves that degraded at a constant rate. Also, it is considered that the chemoattractant is consumed with a rate proportional to the amount of organisms. The cells and chemical substances are transported by a viscous incompressible fluid under the influence of a force due to the aggregation of cells. We prove the existence of global mild solutions in bounded domains of $\mathbb{R}^N,$ $N=2, 3,$ for small initial data in $L^p$-spaces.\vspace{0.3cm}

\vspace{0.3cm}

\end{abstract}

\section{Introduction}
\hspace{0.4cm}Chemotaxis is the oriented movement of cells toward the concentration gradient of certain chemicals in their environment. One of the most interesting phenomena in chemotaxis is the aggregation of chemotactic cells and pattern formation. Chemotactic attraction refers to the movement of cells toward the increasing concentration of a signal, whereas chemotactic repulsion means that cells move along the decreasing concentration of a cue (see for instance \cite{horstmann2011generalizing,liu2013pattern,luca2003chemotactic,Quinlan2005decay} and references therein).  Interactions between cells and the chemical signal may cause several interesting biological patterns. More recent observations show that, in certain cases of chemotactic motion in liquid environments, the mutual interaction between cells and fluid may be substantial (see for instance \cite{chertock2012sinking,dombrowski2004self,hill2005bioconvection,tuval2005bacterial} and references therein). Also, it is important to consider the biological situation which the bacterial population may proliferate according to a logistic law \cite{bellomo2015toward,hillen2009user,hillen2013convergence}; in fact, in several applications, the respective biological setting requires to take into account the proliferation and death of cells, for example, bacterial pattern formation \cite{tyson1999model,woodward1995spatio} or endothelial cell movement and growth in response to a chemical substance known as tumour angiogenesis factor (TAF), which has a significant role in the process of cancer cell invasion of neighboring tissue \cite{chaplain2005mathematical,chaplain1993model,mantzaris2004mathematical}. Previous references lead us to consider the following attraction-repulsion chemotaxis model under the effect of an incompressible viscous fluid with logistic source
\begin{equation}\label{KNS}
\left\{
\begin{array}{lc}
n_{t}+u\cdot\nabla n=\Delta n-\chi \nabla\cdot( n\nabla c)+\xi \nabla\cdot( n\nabla v)+\varsigma n-\mu n^2, &\\
c_{t}+u\cdot\nabla c=\Delta c+\kappa_1(\alpha_1 n- \beta_1 c)-\kappa_2 \gamma cn,  &\\
v_{t}+u\cdot\nabla v=\Delta v+\alpha_2 n-\beta_2 v,  &\\
u_{t}+(u\cdot\nabla)u=\Delta u-\nabla\pi-n \nabla \phi, &\\
\nabla\cdot u=0, &
\end{array}
\right.
\end{equation}
in $\Omega\times(0,T),$ where $\Omega$ is a (sufficiently smooth) bounded domain of $\mathbb{R}^N,$ $N=2,3,$ and $0<T\leq\infty$ is an arbitrary existence time. Here $n = n(x, t),$ $c =c(x, t),$ $v = v(x, t),$ $\pi(x,t)$ and $u(x,t)$ denote respectively the cell density, the concentration of an attractive chemical signal, the concentration of a repulsive chemical signal, the hydrostatic pressure, and the velocity field of the fluid at position $x\in \Omega$ and time $t\in(0,T)$. The evolution of the velocity field $u(x,t)=\left[u_{1}(x,t),..,u_{N}(x,t)\right]$ is governed by the incompressible Navier-Stokes system. In $(\ref{KNS})_4$ the term $-n \nabla \phi$ is a force due to aggregation which arises from the density difference between the fluid with and without the presence of organisms. This term is obtained from an approximation in the fluid, similar to the Boussinesq one for convection problems \cite{chandrasekhar2013hydrodynamic}, which establishes that the effects due to density variations caused by cell aggregation appear only in the buoyancy forcing. The term $-\nabla\cdot(\chi n\nabla c)$ reflects the attractive movement of cells, whereas the term $\nabla\cdot(\xi n\nabla v)$ represents the repulsion migration. In the second equation of (\ref{KNS}), $\kappa_1, \kappa_2 \in \left\{0,1\right\}$ allow two kind of interactions: 1$)$ If $\kappa_1=1$ the attractive signal is produced by the cells  themselves and degrade at a constant rate. 2$)$ If $\kappa_2=1$ the chemical is consumed with a rate proportional to the amount of organisms. The parameters $\chi,$ $\xi,$ $\alpha_1,$ $\beta_1,$ $\alpha_2,$ $\beta_2$ and $\gamma$ are positive constants that represent the chemotactic behavior; with more details, $\chi$ and $\xi$ denote the chemotactic coefficients, $\alpha_1$ and $\alpha_2$ represent the chemical production rate, $\gamma$ represents the chemical consumed rate and, $\beta_1$ and $\beta_2$ denote the chemical degradation rates; finally, $\varsigma,$ and $\mu$ are non-negative constants which describe the organism growth rate and the carrying capacity (i.e. the maximum sustainable population), respectively.\\

System (\ref{KNS}) is completed with the following initial data and boundary conditions
\begin{equation}\label{initialdata}
\left\{
\begin{array}{lc}
\left[n(x,0),c(x,0),v(x,0),u(x,0)\right]=\left[n_0(x),c_0(x),v_0(x),u_0(x)\right],\ x\in\Omega,\\[.3cm]
\frac{\partial n(x,t)}{\partial \nu}=\frac{\partial c(x,t)}{\partial \nu}=\frac{\partial v(x,t)}{\partial \nu}=0, \quad u(x,t)=0,\quad x\in\partial\Omega,\quad t\in(0,T).
\end{array}
\right.
\end{equation}
System (\ref{KNS})-(\ref{initialdata}) consists of three extensions of the classical Keller-Segel model. The first one is the chemotaxis fluid model, which is known as a challenging model; the second extension corresponds to the attractive-repulsive chemotaxis framework, and the third one corresponds to the chemotaxis models with logistic source. In fact, from the mathematical point of view, system (\ref{KNS})-(\ref{initialdata}) seems not to have been studied yet. However, the issues of existence and long-time behavior of solutions related to the three aforementioned chemotaxis frameworks have attracted the attention of many authors, especially in the last five years; for instance, we mention some works \cite{braukhoff2016global,espejo2015reaction,jins2015boundedness,jin2015asymptotic,jin2016boundedness,kozono2016existence,lankeit2016long,li2016large,li2016boundedness,li2016attraction,lin2016boundedness,liu2012classical,wang2016boundedness,zhang2015attraction,zheng2016boundedness}, which will be briefly reviewed in the sequel. \\

$-$ {\it Attraction-repulsion chemotaxis model without fluid and logistic source}\newline
 System (\ref{KNS}) without fluid and logistic source ($u=0,\varsigma=\mu=\kappa_2=0,\kappa_1=1$) has been analyzed by several authors, see for instance \cite{jins2015boundedness,jin2015asymptotic,jin2016boundedness,lin2016boundedness,liu2012classical}. This model corresponds to a direct generalization of the classical Keller-Segel chemotaxis system; however, the analysis of the large time behavior of solutions is difficult due to the lack of a Lyapunov functional. For one-dimensional bounded domains, Zhi-An Wang \textsl{et al.} \cite{liu2012classical} proved the existence of global classical solutions based on Amann's theory and the method of energy estimates. For a more recent result in one dimension, see \cite{jin2015asymptotic}. For two-dimensional bounded domains, Jin and Wang \cite{jin2016boundedness} considered a parabolic-parabolic-elliptic case where the third equation is replaced by $0=\Delta v+\alpha_2 n-\beta_2 v$. By assuming that the repulsion prevails over the attraction in the sense that $\xi \alpha_2 -\chi \alpha_1 \geq 0$, they proved the existence of a unique classical global solution. For the corresponding 2D parabolic-parabolic-parabolic case, existence of global classical solution in bounded domains was obtained by Jin \cite{jins2015boundedness} under the condition of strict prevalence of the repulsion $\xi \alpha_2 -\chi \alpha_1 >0$. Still considering the fully parabolic case and relying on a new entropy-type inequality, Liu and Tao \cite{liu2015global} proved existence and boundedness of global solutions for initial data $[n_0,c_0,v_0]\in C(\bar{\Omega})\times W^{1,\infty}(\Omega)\times W^{1,\infty}(\Omega)$ and $\Vert n_0\Vert_{L^1}$ small with respect to $\frac{1}{\chi\alpha_1}$. For $N \geq 3$, Jin \cite{jins2015boundedness} obtained the existence of global weak solutions in the class $L^{\frac{5}{4}}((0,T);W^{1,\frac{5}{4}}(\Omega))$ ($T>0$) under the condition $\xi \alpha_2 -\chi \alpha_1 >0$. The existence of global classical solutions for $N\geq3$ is still open. Finally, we refer the reader to \cite{lin2016boundedness} for results about attraction-repulsion chemotaxis models with nonlinear diffusions.\\

$-$ {\it Attraction-repulsion chemotaxis model with logistic source but without fluid}\newline
If $\kappa_2=0,\kappa_1=1$ and $u=0,$ some resent results are known (cf. \cite{li2016large,li2016boundedness,li2016attraction,wang2016boundedness, zhang2015attraction,zheng2016boundedness}). For the parabolic-elliptic-elliptic case, Zhang and Li \cite{zhang2015attraction} proved the existence of global classical solution in bounded domains of $\mathbb{R}^N$, $N\geq 1$, with bounded uniformly continuous initial data, a growth restriction on the logistic source and suitable assumptions on the parameters. They also showed existence of global weak solutions when the logistic damping effect is rather mild. In the fully parabolic case, for bounded domains of $\mathbb{R}^N,$ $N=1,2,$ Li \textsl{et al.} \cite{li2016boundedness,li2016attraction} showed the existence of a unique classical global solution. On the other hand, Y. Wang \cite{wang2016boundedness} obtained the existence of 3D-classical global solutions when the term $\Delta n$ is replaced by $\nabla \cdot ((n+1)^{m-1} \nabla n)$ where either $m>4/3$ and the growth term is given by $\varsigma n-\mu n^\theta$ with $\theta\in(1,2)$ or $m \geq 1$ and the growth term is $\varsigma n-\mu n^2$ but assuming a specific condition on $\mu$. More recently, considering different types of conditions on the parameters and initial data, further results about existence of global solutions can be found in \cite{li2016large, shi2017boundedness, wu2017global, zheng2016boundedness}. \\

$-$ {\it Chemotaxis models with fluid related to (\ref{KNS})}\newline
If $v=0,\kappa_2=\varsigma=0,\kappa_1=1,$ and considering the case of Stokes equation, Espejo and Suzuki \cite{espejo2015reaction} proved the existence of a global weak solution in the class $L^{2}((0,T);H^{1}(\Omega))$ where $\Omega$ is bounded domains or the whole space $\mathbb{R}^{2}.$ In \cite{tao2015boundedness}, the authors considered the three-dimensional chemotaxis-Stokes system with $v=0,\kappa_2=0, \varsigma\geq0,\kappa_1=1$ and proved the existence of a classical global solution under the explicit condition $\mu\geq23.$ If the chemotaxis-fluid interaction is via the Navier-Stokes equations, to the best of our knowledge, there are few results in the literature  \cite{braukhoff2016global,lankeit2016long}. If $v=0,\kappa_1=1,\kappa_2=0,$ Braukhoff \cite{braukhoff2016global} introduces an exchange of oxygen between the fluid and its environment, which leads to a different boundary conditions to (\ref{initialdata}). Then, by requiring sufficiently smooth initial data, it was proved the existence of a unique global classical solution for $N=2$, as well as the existence of a global weak solution for $N=3.$ Results of weak solutions for the three dimensional case were obtained by Lankeit \cite{lankeit2016long}. On the other hand, for attraction-attraction chemotaxis model with fluid without logistic source ($\varsigma=\mu=\kappa_1=0,$ $\kappa_2=1$ and $\xi < 0$), Kozono \textsl{et al.} \cite{kozono2016existence} proved the existence of global mild solutions in the whole space $\mathbb{R}^{N}, N\geq 2,$ with small initial data in weak $L^p$-spaces.\\

The purpose of this paper is to analyze the existence of solutions for system (\ref{KNS})-(\ref{initialdata}) in a framework based on $L^p$-spaces (see definitions of spaces in (\ref{spaceX}), (\ref{spaceY}) and (\ref{spaceYexp})). The spaces employed here are inspired by ones used to study Navier-Stokes equations (see, e.g., Kato \cite{Kato1}). The novelty of model (\ref{KNS}) with respect to the previous references is summarized in the following aspects: from the physical point of view, we are considering an attraction-repulsion chemotaxis phenomenon, where the attractive signal is produced by the cells themselves and degrade at a constant rate or the attractive signal is consumed with a rate proportional to the amount of present organisms; the repulsive signal is produced by cells themselves and also degrade at a constant rate. The organisms and chemical substances are transported by an incompressible viscous fluid under the influence of a force due to the aggregation of cells; it is assumed that the cell density may proliferate following a logistic law, allowing the borderline case $\varsigma=0$ which reflects that cell proliferation is ignored. We also are able to consider the case $\mu=0$ and $\varsigma<0$ which corresponds to the case where there is no birth of cells and the death of organisms occurs at a constant decay rate. From the mathematical point of view, we are considering a larger class of initial data for chemotaxis type models with logistic term (with or without fluid) (see (\ref{spaceX})) in comparison to \cite{braukhoff2016global,espejo2015reaction,lankeit2016long,tao2015boundedness,li2016large,li2016boundedness,li2016attraction,wang2016boundedness, zhang2015attraction,zheng2016boundedness}, namely initial data in $L^p(\Omega)-$spaces; in particular, we are able to consider non-continuous data. Observe that the system (\ref{KNS}) does not possess scaling which makes this system more awkward than related problems such as the Navier-Stokes system, the semilinear heat equation with nonlinearities of type $u^p,$ or even several chemotaxis models as mentioned  before. In this paper, we achieve suitable norms with which, despite having no scaling relation, it is possible to obtain the existence of global mild solutions. Having identified the adequate function spaces, we perform estimates on $L^p$-spaces for integral operators appearing in the mild formulation (see integral equations (\ref{mapIE}) below). In order to estimate those operators, we need to use the decay properties of the Stokes and Neumann heat semigroups. Properties of time decay of the Stokes semigroup are well-known; in the case of heat semigroup with Neumann boundary conditions we need to obtain slightly sharper estimates than those previously found in \cite{winkler2010aggregation} (see Lemma \ref{heat_kernel2} and \ref{heat_kernel2z}). However, some of these estimates require the zero mean condition which carry out some difficulties when dealing with the nonlinear terms $\kappa_2 \gamma cn$ and $\varsigma n-\mu n^2$ in (\ref{KNS}). In order to  overcome this obstacle, we employ the quotient space $L^p/\sim$ of equivalence classes of functions in $L^p$-spaces whose difference is a constant. Then, we consider the Neumann heat semigroup on $L^p/\sim$ and use its obtained decay properties. The existence of mild solutions is obtained through an iterative approach that provides a Cauchy sequence that converges to the solution.

The plan of this paper is as follows. In Section 2, we give some preliminaries, prove estimates for the Neumann heat semigroup in our setting, and state our existence-uniqueness results. Finally, in Section 3, we prove our main results.

\section{Functional spaces and main results}
\hspace{0.4cm}Before stating our results, we introduce some functional spaces. Let
$C_{0,\sigma}^{\infty}(\Omega)$ denote the set of all $C^{\infty}$-
real vector functions ${\varphi}=(\varphi_{1},...,\varphi_{N})$ with
compact support in
$\Omega,$ such that div ${\varphi}=0.$ The closure of ${C}_{0,\sigma}%
^{\infty}$ with respect to norm $\Vert\cdot\Vert_{p}$ of space
$ {L}^{p},$ $1<p<\infty,$ is denoted by
${L}_{\sigma}^{p}(\Omega)^N$. By simplicity in the notation, we will not distinguish between vector and scalar functions; so, we denote ${L}_{\sigma}^{p}(\Omega)^N$ simply by ${L}_{\sigma}^{p}(\Omega)$ and so on. Let us recall the Helmholtz
decomposition $({L}^{p}(\Omega))={L}_{\sigma
}^{p}(\Omega)\oplus{G}^{p}(\Omega),\ 1<p<\infty,$ where ${G}^{p}%
(\Omega)=\{\nabla f\in{L}^{p}(\Omega):\,f\in
L_{loc}^{p}(\overline{\Omega})\}$ (cf. \cite{fujiwara1977l_r}). $\mathbb{P}_{p}$ denotes the projection operator from
${L}^{p}(\Omega)$ onto ${L}_{\sigma}^{p}(\Omega).$
The Stokes operator $A_{p}$ on $L_{\sigma}^{p}$ is defined by ${A}%
_{p}=-\mathbb{P}_{p}\Delta$ with domain $D({A}_{p})=\{{u}\in{W}^{2,p}%
(\Omega):\,{u}|_{\partial\Omega}={0}\}\cap{L}_{\sigma}^{p}.$ It
is well known that $-{A}_{p}$ generates a uniformly bounded analytic
semigroup $\{e^{-t{A}_{p}}\}_{t\geq0}$ of class $C_{0}$ in
${L}_{\sigma}^{p}.$

We also consider the heat semigroup under Neumann boundary conditions, i.e., the Neumann heat semigroup. The operator $\Delta_{p}$ with domain  $D({\Delta}_{p})=\{{u}\in{W}^{2,p} (\Omega):\,\frac{\partial u}{\partial\nu}|_{\partial\Omega}={0}\}$ also
generates a uniformly bounded analytic semigroup $\{e^{t\Delta_{p}}\}_{t\geq0}$ of class $C_{0}$ in ${L}^{p}$ (cf. \cite[Chapter 3]{lunardi2012analytic}). Properties of time decay for the Neumann heat semigroup will be discussed below.

Applying the operator projection $\mathbb{P}$ to the equation (\ref{KNS})$_{4}$, we can treat the problem
(\ref{KNS}) as the following problem of parabolic type in $\Omega\times (0,T)$:
\begin{equation}\label{KNS2}
\left\{
\begin{array}{lc}
n_{t}+u\cdot\nabla n=\Delta n-\chi \nabla\cdot( n\nabla c)+\xi \nabla\cdot( n\nabla v)+\varsigma n-\mu n^2 \!, &\\
c_{t}+u\cdot\nabla c=\Delta c+\kappa_1(\alpha_1 n- \beta_1 c)-\kappa_2 \gamma cn,  &\\
v_{t}+u\cdot\nabla v=\Delta v+\alpha_2 n-\beta_2 v,  &\\
u_{t}+\mathbb{P}(u\cdot\nabla)u=-A u-\mathbb{P}(n \nabla\phi). &
\end{array}
\right.
\end{equation}

As usual, we use Duhamel's principle in order to introduce the integral formulation associated with the system (\ref{KNS2})-(\ref{initialdata}):

\begin{equation}\label{mapIE}
\left\{
\begin{array}{lc}
n(x,t)=e^{\varsigma t}e^{t\Delta}n_{0}-{\displaystyle \int_{0}^{t}e^{\varsigma (t-\tau)}e^{(t-\tau)\Delta}(u\cdot\nabla n+\mu n^2)(\tau)d\tau}\\[.5cm]
\tab-{\displaystyle \int_{0}^{t} e^{\varsigma (t-\tau)}e^{(t-\tau)\Delta} \left(\vphantom{n^{(k)}} \nabla\cdot(\chi n \nabla c-\xi n \nabla v) \right) (\tau) d\tau} , \\[.5cm]
c(x,t)=e^{-\kappa_1 \beta_1 t}e^{t\Delta}c_{0}-{\displaystyle \int_{0}^{t}e^{-\kappa_1 \beta_1(t-\tau)}e^{(t-\tau)\Delta}(u\cdot\nabla c-\kappa_1 \alpha_1 n + \kappa_2 \gamma cn)(\tau)d\tau} , \\[.5cm]
v(x,t)=e^{-\beta_2 t}e^{t\Delta}v_{0}-{\displaystyle\int_{0}^{t}e^{-\beta_2(t-\tau)}e^{(t-\tau)\Delta}(u\cdot\nabla v-\alpha_2 n)(\tau)d\tau} , \\[.5cm]
u(x,t)=e^{-t A}u_{0}-{\displaystyle \int_{0}^{t} e^{-(t-\tau) A}\mathbb{P}(u\cdot\nabla u+n \nabla\phi)(\tau)d\tau}.
\end{array}
\right.
\end{equation}

In order to obtain global existence for (\ref{KNS2})-(\ref{initialdata}) we need some preliminaries results on the asymptotics of the Stokes semigroup and heat semigroup under Neumann boundary conditions. Properties of time decay of the Stokes semigroup are well-known. However, in the case of heat semigroup with Neumann boundary conditions we could not find a complete reference that includes all that is necessary for our analysis; then some of the estimates we use below are slightly sharper than those found in \cite{cao2016global,winkler2010aggregation}.\\

It is known that if $\Omega$ is a bounded domain of $\mathbb{R}^N,$ the Green function $G(x,t;\xi,\tau)$ associated to the heat equation with Neumann boundary condition can be expressed through an eigenfunctions expansion. Consider the eigenvalue problem

\begin{equation*}\label{EVP}
\left\{
\begin{array}{rc}
-\Delta \Psi & = \ \lambda \Psi\  \text{in }\ \Omega ,\\
\frac{\partial \Psi}{\partial \nu} & =  \ 0\  \text{on }\ \partial \Omega .
\end{array}
\right.
\end{equation*}
The eigenvalues are non-negative and zero is an eigenvalue to which corresponds a constant eigenfunction. Denoting by $\left\{ \lambda_i \right\}_{i=0}^{\infty}$ the increasing sequence of eigenvalues and by $\left\{ \Psi_i \right\}_{i=0}^{\infty}$ the orthonormal eigenfunctions, we have that
$$G(x,t;y,\tau)=\sum_{i=0}^{\infty} \Psi_{i}(x) \Psi_{i}(y) e^{-\lambda_i (t-\tau)} . $$
Therefore, the solution of the heat equation with no-flux boundary condition and initial data $w$ can be expressed as
\begin{equation}\label{solgreen}
u(x,t) = \left\langle w,\Psi_{0}\right\rangle \Psi_{0} + \sum_{i=1}^{\infty}\left\langle w,\Psi_{i}\right\rangle \Psi_{i} e^{-\lambda_{i} t}.
\end{equation}
In particular, if $\int_{\Omega}w=0$, the first term in the right side of (\ref{solgreen}) vanishes.
\begin{lemm}\label{green1}
Let $1\leq q\leq p \leq \infty$ and $\rho_1=\inf \{\lambda_i : i \in \mathbb{N}\}.$ There exists $C_0=C_0(\Omega,p,q) >0$ such that
\begin{equation*}
\left\| \int_{\Omega} G(x,t;y,\tau) w (y) d y \right\|_{p} \leq C_0 (t-\tau)^{-\frac{N}{2}(\frac{1}{q}-\frac{1}{p})} e^{-\rho_1 (t-\tau)} \left\| w \right\|_{{q}}\!,
\end{equation*}
for all $0\leq \tau <t$ and $w \in L^{q}(\Omega)$ satisfying $\int_\Omega w=0.$
\end{lemm}

\begin{proof}
First of all, we recall the following pointwise estimate (cf. \cite[Theorem~2.2]{mora1983semilinear})
\begin{equation}\label{green}
G(x,t;y,\tau) \leq \frac{(t-\tau)^{-\frac{N}{2}}}{2^{N} \pi^{\frac{N}{2}}} e^{\frac{-\left|x-y \right|^2}{4(t-\tau)}}, \quad \forall t > \tau, \quad x,y \in \Omega.
\end{equation}
Let $1\leq l \leq \infty$ be such that $\frac{1}{p}=\frac{1}{l}+\frac{1}{q}-1.$ Using (\ref{green}) and the Young inequality, we get
\begin{eqnarray}\label{green2}
\left\| \int_{\Omega} G(x,t;y,\tau) w (y) d y \right\|_{{p}} &\leq & \left\| \int_{\mathbb{R}^N} (t-\tau)^{-\frac{N}{2}} e^{\frac{-\left|x-y \right|^2}{4(t-\tau)}} \left|\mathbf{1}_{\Omega}(y) w(y) \right|  d y \right\|_{L^{p}(\mathbb{R}^N)}  \nonumber \\
&\leq & e^{-\rho_1 (t-\tau)} \| (t-\tau)^{-\frac{N}{2}} e^{\frac{-\left|x \right|^2}{4(t-\tau)}} \|_{L^{l}(\mathbb{R}^N)}  \left\| \mathbf{1}_{\Omega} w  \right\|_{L^{q}(\mathbb{R}^N)} \nonumber \\
&= &C_0 (t-\tau)^{-\frac{N}{2}+\frac{N}{2l}} e^{-\rho_1 (t-\tau)} \left\|w\right\|_{q} \nonumber \\
&= &C_0 (t-\tau)^{-\frac{N}{2}(\frac{1}{q}-\frac{1}{p})} e^{-\rho_1 (t-\tau)} \left\|w\right\|_{q}\!. \nonumber \\
\end{eqnarray}
\end{proof}

\begin{remark}\label{green3}
Let $\left\{ e^{t\Delta} \right\}_{t \geq 0}$ be the Neumann heat semigroup and $1 \leq q\leq p \leq \infty$. In the semigroup notation $e^{t \Delta} w = \int_{\Omega} G(x,t;y,0) w(y)dy$, Lemma \ref{green1} implies that
\begin{equation}\label{e10}
\left\Vert e^{t\Delta} w \right\Vert _{p}\!   \leq C_0 t^{-\frac{N}{2}(\frac{1}{q}-\frac{1}{p})} e^{-\rho_1 t} \left\Vert w\right\Vert _{q}\!,
\end{equation}
for all $t>0$ and $w \in L^{q}(\Omega)$ satisfying $\int_\Omega w=0.$
\end{remark}

\begin{lemm} \label{heat_kernel2}
Let $\left\{ e^{t\Delta} \right\}_{t \geq 0}$ be the Neumann heat semigroup in $\Omega$ and $\rho_1=\inf \{\lambda_i : i \in \mathbb{N}\}$. Then, there exist positive constants $C_1,$ $C_2$ and  $C_3$ such that:
\begin{enumerate}[label=(\roman*)]
\item For $1 \leq p \leq \infty$ it holds

\begin{equation} \label{heat_o1}
\left\Vert \nabla e^{t\Delta} w \right\Vert _{p}\!   \leq C_1 t^{-\frac{1}{2}}\left\Vert w\right\Vert _{p}\!,
\end{equation}
for all $t>0$ and $w \in L^{p}(\Omega).$
\item For $1 \leq q \leq p \leq  \infty$ it holds

\begin{equation} \label{heat_o2}
\left\Vert \nabla e^{t\Delta} w \right\Vert _{p}\!   \leq C_2 t^{-\frac{N}{2}(\frac{1}{q}-\frac{1}{p})-\frac{1}{2}} e^{-\rho_1 t} \left\Vert w\right\Vert _{q}\!,
\end{equation}
for all $t>0$ and $w \in L^{q}(\Omega).$

\item For $1<q\leq p <\infty$ or  $1<q< p\leq \infty$ it holds
\begin{equation}\label{heat_o3}
\left\Vert e^{t\Delta} \nabla \cdot w \right\Vert _{p}\!   \leq C_3 t^{-\frac{N}{2}(\frac{1}{q}-\frac{1}{p})-\frac{1}{2}} e^{-\rho_1 t} \left\Vert w\right\Vert _{q}\!,
\end{equation}
for all $t>0$ and $w \in (L^{q}(\Omega))^N.$
\end{enumerate}

\end{lemm}

\begin{proof} \textit{(i)}  Let $L=-\Delta$ with $D(L)=\{u\in W^{2,p}(\Omega): \frac{\partial u}{\partial \nu}=0\ \mbox{on}\ \partial\Omega\}.$ Then, for $1 \leq p \leq \infty$ and $\omega\in(0,\pi/2)$, the operator $(\lambda+L)^{-1}$ is bounded in $L^p$ for $\lambda\in\mathbb{C}\setminus \{0\}$ with $\vert\mbox{arg}\lambda\vert\leq \pi-\omega,$ and the following estimates hold (cf. \cite{angiuli2010analytic,lunardi2012analytic})
\begin{eqnarray}\label{dun1}
\Vert (\lambda+L)^{-1}u\Vert_p\leq C_p\Vert u\Vert_p/\vert \lambda\vert\ \ \mbox{and} \ \ \Vert \nabla^2(\lambda+L)^{-1}u\Vert_p\leq C_p\Vert u\Vert_p,
\end{eqnarray}
for all $u\in L^p(\Omega)$. By interpolation and (\ref{dun1}), we obtain
\begin{eqnarray}\label{dun2}
\Vert \nabla(\lambda+L)^{-1}u\Vert_p\leq C_p\Vert u\Vert_p/\vert \lambda\vert^{1/2},
\end{eqnarray}
for all $u\in L^p(\Omega).$
In order to obtain the estimate (\ref{heat_o1}), we compute the gradient of the Dunford integral as follows
\begin{eqnarray*}
\nabla e^{-tL}w=\frac{1}{2\pi i}\int_{\Gamma}\nabla e^{\lambda t}(\lambda+L)^{-1},
\end{eqnarray*}
where the path $\Gamma=\Gamma_0\cup\Gamma_{\pm},$ with  $\Gamma_{\pm}:\vert \lambda\vert e^{\pm i\varphi},$ $\frac{1}{t}\leq \vert \lambda\vert,$ and $\Gamma_0:(\frac{1}{t})e^{i\mbox{arg}\lambda},$ $-\varphi\leq \mbox{arg}\lambda\leq \varphi$. Now, using (\ref{dun2}), we arrive at
\begin{eqnarray*}
\Vert \nabla e^{-tL}w\Vert_p=\left\Vert \frac{1}{2\pi i}\int_{\Gamma}\nabla e^{\lambda t}(\lambda+L)^{-1}w\right\Vert_p\leq \frac{C}{2\pi}\int_{\Gamma} e^{\lambda t}\vert \lambda\vert^{-1/2}\Vert w\Vert_p\leq C_{1}t^{-1/2}\Vert w\Vert_p.
\end{eqnarray*}
\textit{(ii)}
We write $\bar{w}:=\frac{1}{\left|\Omega\right|}\int_{\Omega}w$ and thus $\int_{\Omega}(w-\bar{w})=0$. Then, it follows from (\ref{e10}) and (\ref{heat_o1}) that
\begin{eqnarray}\label{neu2}
\left\| \nabla e^{t \Delta} w\right\|_{p} = \left\| \nabla e^{\frac{t}{2} \Delta} e^{\frac{t}{2} \Delta} (w-\bar{w}) \right\|_{p}
&\leq & C_1 t^{-\frac{1}{2}}\| e^{\frac{t}{2} \Delta} (w-\bar{w}) \|_{p} \\ \nonumber
&\leq & C t^{-\frac{N}{2}(\frac{1}{q}-\frac{1}{p})-\frac{1}{2}} e^{-\rho_1 t} \left\|(w-\bar{w})\right\|_{q} \\ \nonumber
&\leq & C_2 t^{-\frac{N}{2}(\frac{1}{q}-\frac{1}{p})-\frac{1}{2}} e^{-\rho_1 t} \left\|w\right\|_{q}\!.
\end{eqnarray}
\textit{(iii)} First we consider the case $1 < q \leq p < \infty .$ Let $\varphi \in C_{0}^{\infty}(\Omega).$ Recalling that $e^{t \Delta}$ is self-adjoint in $L^2,$ integrating by parts and using (\ref{heat_o2}), we get
\begin{eqnarray}\label{neu3}
 \left|\int_{\Omega} e^{t \Delta} \nabla \cdot w \varphi \right| &= & \left| -\int_{\Omega} w \cdot \nabla e^{t \Delta} \varphi \right| \nonumber \\
&\leq & \left\| w \right\|_{q} \left\| \nabla e^{t \Delta} \varphi \right\|_{q'} \nonumber \\
&\leq & C_3 \left\| w \right\|_{q} t^{-\frac{N}{2}(\frac{1}{p'}-\frac{1}{q'})-\frac{1}{2}} e^{-\rho_1 t} \| \varphi \|_{p'}, \nonumber \\
\end{eqnarray}
where $\frac{1}{p}+\frac{1}{p'}=1$ and $\frac{1}{q}+\frac{1}{q'}=1$. Since $\frac{1}{p'}-\frac{1}{q'}=\frac{1}{q}-\frac{1}{p}$, we can complete the first part of the proof by taking the supremum over all $\varphi \in C_{0}^{\infty}(\Omega)$ satisfying $\left\| \varphi \right\|_{p'}\leq 1$. For the case $1 < q < p=\infty$, assume that $w \in (C^{\infty}_{0}(\Omega))^N$. Then $\int_{\Omega}e^{\frac{t}{2}\Delta}\nabla \cdot w = \int_{\Omega} \nabla \cdot w =0.$ Thus, from (\ref{e10}) and the first part, we can estimate

\begin{eqnarray}\label{neu4}
\left\| e^{t \Delta} \nabla \cdot w\right\|_{\infty} &= & \left\| e^{\frac{t}{2} \Delta} (e^{\frac{t}{2} \Delta} \nabla \cdot w) \right\|_{\infty}  \nonumber \\
&\leq &  C_0 t^{-\frac{n}{2q}} e^{-\rho_1 t} \| e^{\frac{t}{2} \Delta} \nabla \cdot w \|_{q} \nonumber \\
&\leq &  C_3 t^{-\frac{n}{2q}} t^{-\frac{1}{2}} e^{-2\rho_1 t} \left\|w\right\|_{q}\!. \nonumber \\
\end{eqnarray}
Finally, by \eqref{neu4} and an argument of density, we obtain (\ref{heat_o3}).

\end{proof}

Some basic estimates of the Stokes semigroup $\{e^{-tA}\}_{t\geq0}$ in $L_\sigma^{p}(\Omega)$ are listed in the following lemma (see \cite{giga1981}).

\begin{lemm} \label{stokes_kernel}
Let $\{e^{-tA}\}_{t\geq0}$ be the Stokes semigroup in $L^{p}_{\sigma}(\Omega)$ and $\rho_2 \in (0,\nu_2),$ where $\nu_2=\inf \operatorname{Re} \sigma(A)>0.$ Then:
\begin{enumerate}[label=(\roman*)]
\item  For $1< q \leq p < \infty$, there exists $C_4=C_4(\Omega,p,q)>0$ such that
$$\left\Vert e^{tA} w \right\Vert _{p}\!   \leq C_4 t^{-\frac{N}{2}(\frac{1}{q}-\frac{1}{p})}e^{-\rho_2 t}\left\Vert w\right\Vert _{q}\!,$$
for all $t>0$ and $w\in L_{\sigma}^{q}(\Omega).$
\item For $1<q\leq p < \infty$, there exists $C_5=C_5(\Omega,p,q)>0$ such that
$$\left\Vert \nabla e^{tA} w \right\Vert _{p}\!   \leq C_5 t^{-\frac{N}{2}(\frac{1}{q}-\frac{1}{p})-\frac{1}{2}} e^{-\rho_2 t} \left\Vert w\right\Vert_{q}\!,$$
for all $t>0$ and $w\in L_{\sigma}^{q}(\Omega).$
\item For $1<q\leq p < \infty$, there exists $C_6=C_6(\Omega,p,q)>0$ such that
$$\left\Vert e^{tA} \nabla \cdot w \right\Vert _{p}\!   \leq C_6 t^{-\frac{N}{2}(\frac{1}{q}-\frac{1}{p})-\frac{1}{2}} e^{-\rho_2 t} \left\Vert w\right\Vert_{q}\!,$$
for all $t>0$ and $w\in (L_{\sigma}^{q}(\Omega))^N.$
\end{enumerate}
\end{lemm}

For $1\leq p\leq \infty$ we consider  the space $L^p/\sim$ of all equivalence classes of functions in $L^p(\Omega)$ whose difference is a constant, that is, if $f\in L^p(\Omega),$ and $\left[ f\right]$ denotes the equivalence class of $f,$ then $g\in \left[ f \right]$ if and only if, $f-g$ is a constant. $L^p/\sim$ is a vector space with the operations $+$ and scalar product defined respectively by $\left[ f \right]+\left[ g \right]=\left[ f+g\right]$ and $a\left[ f \right]=\left[ a f \right],$ $a\in\mathbb{R}.$ The product in $L^p/\sim$ is defined by $\left[ f \right]\left[ g \right]=\left[(f-\bar{f})(g-\bar{g})\right],$ where $\bar{f}=\frac{1}{\vert\Omega\vert}\int_\Omega f$ and $\bar{g}=\frac{1}{\vert\Omega\vert}\int_\Omega g.$  It holds that $L^p/\sim$ is a Banach space with the norm
\begin{eqnarray}\label{norm}
\Vert \left[ f\right]\Vert_{p/\sim}:=\inf\{ \Vert f+c\Vert_{p}\ :\ c\ \mbox{is a constant}\}.
\end{eqnarray}
For fixed $f\in L^p$, note that $c\rightarrow c+f$ is continuous from $\mathbb{R}$ to $L^p$.
Using this fact and that the norm $\Vert\cdot\Vert_p$ is a continuous function in $L^p$,
we can show that for each $[f]\in L^p/\sim$ there exists $f^*\in[f]$ such that $\Vert f^*\Vert_q=\Vert \left[ f\right]\Vert_{q/\sim}$.
Notice that for $w \in L^p/\sim$, it holds that $e^{t\Delta}[w]=[e^{t\Delta}w]$. Thus, $e^{t\Delta}$ can be extended naturally to $L^p/\sim$ by making $e^{t\Delta}[w]=[e^{t\Delta}w]$ for $w \in L^p/\sim.$ Next lemma establishes the $L^p/\sim$ to $L^q/\sim$ decay estimate of $e^{t \Delta}$. 

\begin{lemm} \label{heat_kernel2z}
Let $\left\{ e^{t\Delta} \right\}_{t \geq 0}$ be the Neumann heat semigroup in $\Omega.$ Then
for $1 \leq q \leq p \leq  \infty$ there exists $C_0=C_0(\Omega,p,q)>0$ such that
\begin{equation}\label{e10z}
\left\Vert e^{t\Delta} [w] \right\Vert _{p/\sim}\!   \leq C_0 t^{-\frac{N}{2}(\frac{1}{q}-\frac{1}{p})} e^{-\rho_1 t} \left\Vert [w]\right\Vert _{q/\sim}\!,
\end{equation}
for all $[w]\in L^{q}/\sim.$

\end{lemm}
\begin{proof} For each $w\in L^q(\Omega),$ we write $\bar{w}:=\frac{1}{\left|\Omega\right|}\int_{\Omega}w.$ Then we have that $w-\bar{w}\in [w]$ and $\int_{\Omega}(w-\bar{w})=0.$ Thus, from (\ref{norm}) and (\ref{e10}) we have
\begin{eqnarray}\label{ker18}
\left\Vert e^{t\Delta} [w] \right\Vert _{p/\sim}&=&\inf\{\Vert e^{t\Delta}(w+c)\Vert_{p}\ :\ c\ \mbox{constant}\}\nonumber\\
&\leq & \Vert e^{t\Delta}(w-\bar{w})\Vert_{p}\nonumber\\
&\leq & C_0 t^{-\frac{N}{2}(\frac{1}{q}-\frac{1}{p})} e^{-\rho_1 t} \left\Vert w-\bar{w}\right\Vert _{q}\nonumber\\
&\leq & C_0 t^{-\frac{N}{2}(\frac{1}{q}-\frac{1}{p})} e^{-\rho_1 t} (\left\Vert w+c\right\Vert _{q}+\left\Vert \bar{w}+c\right\Vert _{q})\nonumber\\
&\leq & C_0 t^{-\frac{N}{2}(\frac{1}{q}-\frac{1}{p})} e^{-\rho_1 t} (\left\Vert w+c\right\Vert _{q}+\vert\Omega\vert^{\frac{1}{q}-1}\left\Vert \bar{w}+c\right\Vert _{1})\nonumber\\
&\leq & C_7 t^{-\frac{N}{2}(\frac{1}{q}-\frac{1}{p})}e^{-\rho_1 t} \left\Vert w+c\right\Vert _{q}.
\end{eqnarray}
Then, taking the infimum over the set of constants and noting that $[w-\bar{w}]=[w]$, we get
\begin{eqnarray}
\left\Vert e^{t\Delta} [w] \right\Vert _{p/\sim}
\leq  C_0 t^{-\frac{N}{2}(\frac{1}{q}-\frac{1}{p})} e^{-\rho_1 t} \Vert [{w}]\Vert _{q/\sim}.\nonumber
\end{eqnarray}

\end{proof}
Now we introduce suitable time-dependent functional spaces to
study the initial value problem (\ref{KNS2})-(\ref{initialdata}). For $N<r\leq \infty$ and $0<T\leq\infty,$ we define the Banach space $\mathcal{X}=\mathcal{X}_{r}^T$ of initial data by
\begin{equation}\label{spaceX}
 \begin{aligned}
\mathcal{X} := & \left\{ \vphantom{\sup_{0<t<T}} \left[ [n_{0}],[c_{0}],[v_{0}],u_{0}\right] \in (L^{\frac{N}{2}}/\sim)\times (L^\infty/\sim)\times (L^\infty/\sim)\times (L_\sigma^{N}): \right.\\
&\ \ \ \ \left. \sup_{0<t<T} t^{\frac{N}{2}(\frac{1}{N}-\frac{1}{r})} \left\Vert \nabla e^{t\Delta}c_{0}\right\Vert _{r}< \infty , \quad
 \sup_{0<t<T} t^{\frac{N}{2}(\frac{1}{N}-\frac{1}{r})} \left\Vert \nabla e^{t\Delta}v_{0}\right\Vert _{r}< \infty \right\}\!,
 \end{aligned}
\end{equation}
with the norm $
\left\Vert \left [[n_{0}],[c_{0}],[v_{0}],u_{0}\right]\right\Vert _{\mathcal{X}} :=  \left\Vert [n_{0}]\right\Vert _{\mathcal{X}_1}+\left\Vert [c_{0}]\right\Vert _{\mathcal{X}_2}+\left\Vert [v_{0}]\right\Vert _{\mathcal{X}_3}+\left\Vert u_{0}\right\Vert _{\mathcal{X}_4}\!,
$
where

\begin{equation*}\label{normX}
\left\{
\begin{array}{lc}
\left\Vert [n_{0}]\right\Vert _{\mathcal{X}_1}:=\left\Vert [n_{0}]\right\Vert _{\frac{N}{2} /\sim }\!,\ \quad
\left\Vert [c_{0}]\right\Vert _{\mathcal{X}_2}:=\left\Vert [c_{0}]\right\Vert _{\infty/\sim}+ {\displaystyle \sup_{0<t<T} t^{\frac{N}{2}(\frac{1}{N}-\frac{1}{r})} \left\Vert \nabla e^{t\Delta}c_{0}\right\Vert _{r}\!,}\\[0.4cm]
\left\Vert [v_{0}]\right\Vert _{\mathcal{X}_3}:= \left\Vert [v_{0}]\right\Vert _{\infty/\sim} + {\displaystyle \sup_{0<t<T} t^{\frac{N}{2}(\frac{1}{N}-\frac{1}{r})} \left\Vert \nabla e^{t\Delta}v_{0}\right\Vert _{r}\!,}\ \quad
\left\Vert u_{0}\right\Vert _{\mathcal{X}_4}:=\left\Vert u_{0}\right\Vert _{N}\!.\
\end{array}
\right.
\end{equation*}
For $N<p,r\leq \infty$ and $N/2<q\leq\infty$, we consider the Banach spaces $\mathcal{Y}=\mathcal{Y}_{p,r,q}^T$ and $\mathcal{Y}^{\tiny{\mbox{exp}}}=\mathcal{Y}^{\tiny{\mbox{exp}}}_{p,r,q}$ defined by
\begin{equation}\label{spaceY}
\begin{aligned}
\mathcal{Y} := & \left\{ \vphantom{n^{(k)}} [[n],[c],[v],u]: t^{\frac{N}{2}(\frac{2}{N}-\frac{1}{q})}n\in BC(\left[0,T\right);L^{q}),\ c\in L^{\infty}(\left[0,T\right);L^{\infty}), \right.\\
&\ \ \ \  t^{\frac{N}{2}(\frac{1}{N}-\frac{1}{r})}\nabla c\in BC(\left[0,T\right);L^{r}),\ v\in L^{\infty}(\left [0,T\right);L^{\infty}), \\
&\ \ \  \left. t^{\frac{N}{2}(\frac{1}{N}-\frac{1}{r})}\nabla v\in BC(\left[0,T\right);L^{r}),\ t^{\frac{N}{2}(\frac{1}{N}-\frac{1}{p})}u\in BC(\left[0,T\right);L_\sigma^{p}) \right\}\!
\end{aligned}
\end{equation}
and
\begin{equation}\label{spaceYexp}
\begin{aligned}
\mathcal{Y}^{\tiny{\mbox{exp}}}:= & \left\{ \vphantom{n^{(k)}} [[n],[c],[v],u]: e^{\tilde{\varsigma}t}t^{\frac{N}{2}(\frac{2}{N}-\frac{1}{q})}n\in BC(\left[0,\infty\right);L^{q}),\ e^{\min \{ \kappa_1 \beta_1 ,\tilde{\varsigma} \} t} c\in L^{\infty}(\left[0,\infty\right);L^{\infty}), \right.\\
&\ \ \ \  t^{\frac{N}{2}(\frac{1}{N}-\frac{1}{r})}\nabla c\in BC(\left[0,\infty\right);L^{r}),\ e^{\min \{ \beta_2, \tilde{\varsigma} \} t} v\in L^{\infty}(\left [0,\infty\right);L^{\infty}), \\
&\ \ \  \left. t^{\frac{N}{2}(\frac{1}{N}-\frac{1}{r})}\nabla v\in BC(\left[0,\infty\right);L^{r}),\ e^{\min \{ \rho_2 ,\tilde{\varsigma} \} t} t^{\frac{N}{2}(\frac{1}{N}-\frac{1}{p})} u\in BC(\left[0,\infty\right);L_\sigma^{p}) \right\}\!.
\end{aligned}
\end{equation}
The space $\mathcal{Y}$ is endowed with the norm
\begin{equation*}
\begin{aligned}
\left\Vert \left [[n],[c],[v],u\right] \right\Vert _{\mathcal{Y}} :=  \left\Vert [n]\right\Vert _{\mathcal{Y}_1}+\left\Vert [c]\right\Vert _{\mathcal{Y}_2}+\left\Vert [v]\right\Vert _{\mathcal{Y}_3}+\left\Vert u\right\Vert _{\mathcal{Y}_4}\!,
\end{aligned}
\end{equation*}
where

\begin{equation*}\label{normY}
\left\{
\begin{array}{lc}
\left\Vert [n]\right\Vert _{\mathcal{Y}_1}:= {\displaystyle \sup_{0<t<T} t^{\frac{N}{2}(\frac{2}{N}-\frac{1}{q})}\left\Vert [n(t)]\right\Vert _{q/\sim}} \text{,} \\[0.4cm]
\left\Vert [c]\right\Vert _{\mathcal{Y}_2}:= {\displaystyle \sup_{0<t<T} \left\Vert [c(t)]\right\Vert _{\infty/\sim}+ {\displaystyle \sup_{0<t<T} t^{\frac{N}{2}(\frac{1}{N}-\frac{1}{r})}\left\Vert \nabla c(t)\right\Vert _{r}}} \text{,}\ \\[0.4cm]
\left\Vert [v]\right\Vert _{\mathcal{Y}_3}:= {\displaystyle \sup_{0<t<T} \left\Vert [v(t)]\right\Vert _{\infty/\sim}}+ {\displaystyle \sup_{0<t<T} t^{\frac{N}{2}(\frac{1}{N}-\frac{1}{r})}\left\Vert \nabla v(t)\right\Vert _{r}} \text{,}\ \\
\left\Vert u\right\Vert _{\mathcal{Y}_4}:={\displaystyle \sup_{0<t<T} t^{\frac{N}{2}(\frac{1}{N}-\frac{1}{p})}\left\Vert u(t)\right\Vert _{p}} \text{.}\
\end{array}
\right.
\end{equation*}
The space $\mathcal{Y}^{\tiny{\mbox{exp}}}$ is endowed with the norm
\begin{equation*}
\begin{aligned}
\left\Vert \left [[n],[c],[v],u\right] \right\Vert _{\mathcal{Y}^{\tiny{\mbox{exp}}}} :=  \left\Vert [n]\right\Vert _{\mathcal{Y}^{\tiny{\mbox{exp}}}_1}+\left\Vert [c]\right\Vert _{\mathcal{Y}^{\tiny{\mbox{exp}}}_2}+\left\Vert [v]\right\Vert _{\mathcal{Y}^{\tiny{\mbox{exp}}}_3}+\left\Vert u\right\Vert _{\mathcal{Y}^{\tiny{\mbox{exp}}}_4}\!,
\end{aligned}
\end{equation*}
where

\begin{equation*}\label{normYexp}
\left\{
\begin{array}{lc}
\left\Vert [n]\right\Vert _{\mathcal{Y}^{\tiny{\mbox{exp}}}_1}:= {\displaystyle \sup_{0<t<\infty} e^{\tilde{\varsigma}t} t^{\frac{N}{2}(\frac{2}{N}-\frac{1}{q})}\left\Vert [n(t)]\right\Vert _{q/\sim}} \text{,} \\[0.4cm]
\left\Vert [c]\right\Vert _{\mathcal{Y}^{\tiny{\mbox{exp}}}_2}:= {\displaystyle \sup_{0<t<\infty} e^{\min \{ \kappa_1 \beta_1 ,\tilde{\varsigma} \} t} \left\Vert [c(t)]\right\Vert _{\infty/\sim}+ {\displaystyle \sup_{0<t<\infty} t^{\frac{N}{2}(\frac{1}{N}-\frac{1}{r})}\left\Vert \nabla c(t)\right\Vert _{r}}} \text{,}\ \\[0.4cm]
\left\Vert [v]\right\Vert _{\mathcal{Y}^{\tiny{\mbox{exp}}}_3}:= {\displaystyle \sup_{0<t<\infty} e^{\min \{ \beta_2 ,\tilde{\varsigma} \} t} \left\Vert [v(t)]\right\Vert _{\infty/\sim}}+ {\displaystyle \sup_{0<t<\infty} t^{\frac{N}{2}(\frac{1}{N}-\frac{1}{r})}\left\Vert \nabla v(t)\right\Vert _{r}} \text{,}\ \\
\left\Vert u\right\Vert _{\mathcal{Y}^{\tiny{\mbox{exp}}}_4}:={\displaystyle \sup_{0<t<\infty} e^{\min \{ \rho_2 ,\tilde{\varsigma} \} t}t^{\frac{N}{2}(\frac{1}{N}-\frac{1}{p})}\left\Vert u(t)\right\Vert _{p}} \text{.}\
\end{array}
\right.
\end{equation*}

From now on, for the density $n$ and concentrations functions $c,v$, we will denote abusively these functions in $L^p$ and their equivalence classes in $L^p/\sim$ in the same way. For example, we write $[n,c,v,u]\in\mathcal{Y}$ in place of $[[n],[c],[v],u]\in\mathcal{Y}$. With this convection, now we are in position to establish the notion of solution that we will deal with.
\begin{defi} \label{defimild}
Let $[n_0,c_0,v_0,u_0]\in\mathcal{X}.$ A mild solution for the initial value problem (\ref{KNS})-(\ref{initialdata}) is a vector $[n,c,v,u]\in\mathcal{Y}$ satisfying the integral system (\ref{mapIE}). Thus, the three first integral equations in (\ref{mapIE}) must be understood as equivalence classes. In addition, in the fourth integral equation in (\ref{mapIE}), $n$ can be taken as any representative of $[n]$ since the term $\int_{0}^{t} e^{-(t-\tau) A}n\nabla\phi d\tau$ is invariant in the set $\{\tilde{n};[\tilde{n}]=[n]\}$.
\end{defi}
In what follows, we state our results.
\begin{theo}\label{theo1}
Assume either
\begin{enumerate}[label=(\roman*)]
\item $N=3,$ $N \leq s \leq \infty$, $\frac{N}{2}<q<N$, $N<p<\frac{Nqs}{Ns+Nq-2sq}$ and $N<r<\frac{Nq}{N-q}$,
\end{enumerate}
or
\begin{enumerate}[label=(\roman*)]
\item[\textit{(}ii\textit{)}] $N=2$, $N<s<\infty$, $\frac{s}{s-1}\leq q<N$, $\frac{q}{q-1}\leq p<\infty$ and $N<r<\frac{Nq}{N-q}$,
\end{enumerate}
or
\begin{enumerate}[label=(\roman*)]
\item[\textit{(}iii\textit{)}] $N=2$, $s = \infty$, $\frac{N}{2}< q<N$, $\frac{q}{q-1}\leq p<\infty$ and $N<r<\frac{Nq}{N-q}$.
\end{enumerate}
 Let $0<T<\infty$ be arbitrary, $[n_0,c_0,v_0,u_0]\in\mathcal{X}$ and $t^{\frac{1}{2}-\frac{N}{2s}} \nabla\phi \in BC(\left[0,T\right);L^{s})$. There exists $\delta>0$ such that if $\left\Vert \left [n_{0},c_{0},v_{0},u_{0}\right]\right\Vert _{\mathcal{X}}<\delta,$ then problem (\ref{KNS})-(\ref{initialdata}) has a mild solution $[n,c,v,u]\in \mathcal{Y}$. Such mild solution is unique in a suitable closed ball in $\mathcal{Y}$.
\end{theo}

\begin{remark}
The class of functions where we find the mild solution is settled by the exponents $p,q$ and $r.$ In particular, the ranges for $p$ and $r$ are simultaneously determined in terms of $q$, which suggests that the density of microorganisms has a dominant behavior in relation to the chemicals and fluid, see Remark \ref{rem1a}.
\end{remark}
\begin{remark}\label{rem1}
Notice that the mild solution $[n,c,v,u]\in \mathcal{Y}$ for (\ref{KNS})-(\ref{initialdata}) obtained in Theorem \eqref{theo1} is global in the sense that the fixed $T>0$ can be taken arbitrarily large. An interesting model related to (\ref{KNS}) is when we assumed $\varsigma=0$. This kind of model describes several biological behaviors, for instance, the phenomenon of broadcast spawning (cf. \cite{kiselev2012biomixing}); in this case, the term $-\mu n^2$ describes the reaction (fertilization) term. If $\varsigma=0$, the solution $[n,c,v,u]$ is defined on $[0,\infty)$.
\end{remark}
\begin{remark}
Note that if $\rho_1 \geq \varsigma$ then, with minor modifications, the solution provided by theorem \ref{theo1} are defined on $[0,\infty);$ moreover, the solution has a exponential decay. This is obtained by using the exponential decay of the estimates in lemmas \ref{heat_kernel2}, \ref{stokes_kernel} and \ref{heat_kernel2z} and following the proof of Theorem \ref{theo2} below. In fact, in this case we have that $[n],[c],[v],u$ decay exponentially to zero which means that $n,c,v$ converge toward constants and $u\rightarrow 0$.
\end{remark}

\noindent An interesting situation, not only from the mathematical point of view, but also in several physical situations (e.g. fertilization processes), occurs when we consider that there is no birth of cells and the death of organisms occurs at a constant decay rate. In this case we have the following result:

\begin{theo}\label{theo2} Let us consider in (\ref{KNS}) the term $-\tilde{\varsigma} n$ in place of $\varsigma n-\mu n^2$ with $\tilde{\varsigma}\geq0$. Assume that $T=\infty$, ${[n_0,c_0,v_0,u_0]\in\mathcal{X}}$ and $\nabla\phi \in L^{\infty}([0,\infty); L^{N}).$ If $N=3$ assume that the exponents $p,$ $q$ and $r$ satisfy either $\textit{(}i\textit{)},$ $\textit{(}ii\textit{)}$ or $\textit{(}iii\textit{)}$ below:

\begin{enumerate}[label=(\roman*)]
\item $\frac{N}{2}<q<N, \quad N<p<\frac{Nq}{N-q}, \quad N<r<\frac{Nq}{N-q},$
\item $q=N, \quad N<p<\infty, \quad N<r<\infty,$
\item $N<q<2N, \quad N<p<\frac{Nq}{q-N}, \quad q \leq r<\frac{Nq}{q-N}$.
\end{enumerate}
In the case $N=2$, assume that the exponents $p,$ $q$ and $r$ satisfy one of the above condition $\textit{(}ii\textit{)}$ or $\textit{(}iii\textit{)}$. Then, there exists $\delta>0$ such that problem (\ref{KNS})-(\ref{initialdata}) has a global mild solution $[n,c,v,u]\in \mathcal{Y}^{\tiny{\mbox{exp}}}$ provided that $\left\Vert \left [n_{0},c_{0},v_{0},u_{0}\right]\right\Vert _{\mathcal{X}}<\delta$. This solution is unique in a suitable closed ball in $\mathcal{Y}^{\tiny{\mbox{exp}}}$.
\end{theo}
\begin{remark}
With slight modifications on the proofs and range of the exponents $p$, $q$ and $r$, Theorem \ref{theo2} continues true if we consider $\nabla\phi$ in the class $t^{\frac{1}{2}-\frac{N}{2s}} \nabla\phi \in BC(\left[0,\infty\right);L^{s})$, for $s\geq N$.
\end{remark}

\section{Proof of Theorems \ref{theo1} and \ref{theo2}}
In this section we develop the proofs of the results stated in Section 2.

\begin{proof}[Proof of Theorem \ref{theo1}]

First we estimate each term in the integral system (\ref{mapIE}) in the norm of the functional space $\mathcal{Y}.$\newline
\\
{\underline{Estimates for $n$}}\newline
From Lemma \ref{heat_kernel2z}, we get

\begin{equation}\label{no1}
\left\Vert e^{\varsigma t} e^{t \Delta }n_{0}\right\Vert _{q}\leq Ce^{\varsigma t}t^{-\frac{N}{2}(\frac{2}{N}-\frac{1}{q})}\left\Vert n_{0}\right\Vert _{\frac{N}{2}}\!.
\end{equation}
On the other hand, by the conditions on $p$ and $q$, we have $\frac{1}{p}+\frac{1}{q}\leq 1,$ $\frac{1}{2}-\frac{N}{2p}>0,$ $\frac{N}{2}(\frac{1}{p}+\frac{1}{q})-\frac{1}{2}>0,$ $1-\frac{N}{2q}>0$ and $\frac{N}{q}-1>0$. Thus, it follows from Lemmas \ref{heat_kernel2} and \ref{heat_kernel2z} that

\begin{eqnarray}\label{no2}
&&\left\Vert \int_{0}^{t} e^{\varsigma (t-\tau)} e^{(t-\tau) \Delta}(u\cdot\nabla n+\mu n^{2})(\tau)d\tau\right\Vert _{q} \nonumber \\
&& \ \ \  \leq\int_{0}^{t}\left\Vert e^{\varsigma (t-\tau)} e^{(t-\tau) \Delta}(\nabla\cdot(nu)+\mu n^{2})(\tau)\right\Vert _{q}d\tau \nonumber \\
&&\ \ \ \leq Ce^{\varsigma T}\int_{0}^{t}(t-\tau)^{-\frac{N}{2p}-\frac{1}{2}}\left\Vert n(\tau)\right\Vert _{q}\left\Vert u(\tau)\right\Vert _{p}d\tau \nonumber \\
&& \ \ \ \ \ \ +Ce^{\varsigma T}\int_{0}^{t}(t-\tau)^{-\frac{N}{2q}}\left\Vert n(\tau)\right\Vert _{q}\left\Vert n(\tau)\right\Vert _{q}d\tau \nonumber \\
&&\ \ \  \leq C e^{\varsigma T} t^{-\frac{N}{2}(\frac{2}{N}-\frac{1}{q})} \left(\sup_{0<\tau<T}\tau^{\frac{N}{2}(\frac{2}{N}-\frac{1}{q})}\left\Vert n(\tau)\right\Vert _{q}\right)\left(\sup_{0<\tau<T}\tau^{\frac{N}{2}(\frac{1}{N}-\frac{1}{p})}\left\Vert u(\tau)\right\Vert _{p} \right) \nonumber \\
&&\ \ \ \ \ \ \times B\left (\frac{1}{2}-\frac{N}{2p},\frac{N}{2}(\frac{1}{p}+\frac{1}{q})-\frac{1}{2}\right)\nonumber \\
&&\ \ \ \ \ \ +C e^{ \varsigma T}t^{-\frac{N}{2}(\frac{2}{N}-\frac{1}{q})} \left (\sup_{0<\tau<T}\tau^{\frac{N}{2}(\frac{2}{N}-\frac{1}{q})}\left\Vert n(\tau)\right\Vert _{q}\right)^{2}\times B\left(1-\frac{N}{2q},\frac{N}{q}-1\right)\!,
\end{eqnarray}
\noindent where $C=C(N,p,q)>0$ and $B(\cdot,\cdot)$ denotes the beta function. Moreover, since $\frac{1}{q}+\frac{1}{r}\leq 1,$ $\frac{1}{2}-\frac{N}{2r}>0$ and $\frac{N}{2}(\frac{1}{q}+\frac{1}{r}-\frac{1}{N})>0$, we obtain from Lemma \ref{heat_kernel2} that

\begin{eqnarray}\label{no3}
&&\left\Vert \int_{0}^{t}e^{\varsigma(t-\tau)}e^{(t-\tau)\Delta} \left(\vphantom{n^{(k)}} \nabla\cdot(\chi n\nabla c-\xi n\nabla v) \right) (\tau)d\tau\right\Vert _{q}\nonumber \\
&&\ \ \  \leq Ce^{\varsigma T}\int_{0}^{t}(t-\tau)^{-\frac{N}{2r}-\frac{1}{2}}\left\Vert n(\tau)\right\Vert _{q}(\left\Vert \nabla c(\tau)\right\Vert _{r}+\left\Vert \nabla v(\tau)\right\Vert _{r})d\tau \nonumber \\
&&\ \ \ \leq C e^{\varsigma T} t^{-\frac{N}{2}(\frac{2}{N}-\frac{1}{q})} \left(\sup_{0<\tau<T}\tau^{\frac{N}{2}(\frac{2}{N}-\frac{1}{q})} \left\Vert n(\tau)\right\Vert _{q} \right)\nonumber \\
&&\ \ \ \ \ \ \times \left( \sup_{0<\tau<T}\tau^{\frac{N}{2}(\frac{1}{N}-\frac{1}{r})}\left\Vert \nabla c(\tau)\right\Vert _{r}+\sup_{0<\tau<T}\tau^{\frac{N}{2}(\frac{1}{N}-\frac{1}{r})}\left\Vert \nabla v(\tau)\right\Vert _{r}\right)\nonumber \\
&&\ \ \ \ \ \ \times B\left(\frac{1}{2}-\frac{N}{2r},\frac{N}{2}(\frac{1}{q}+\frac{1}{r}-\frac{1}{N})\right)\!.
\end{eqnarray}
Using (\ref{no1}), (\ref{no2}) and (\ref{no3}), we get

\begin{equation}\label{no4}
\left\|n\right\|_{\mathcal{Y}_1} \leq Ce^{\varsigma T} \left( \left\|n_0\right\|_{\mathcal{X}_1}+\left\|n\right\|_{\mathcal{Y}_1}\left\|u\right\|_{\mathcal{Y}_4}+\left\|n\right\|_{\mathcal{Y}_1}\left\|n\right\|_{\mathcal{Y}_1}+\left\|n\right\|_{\mathcal{Y}_1}\left\| c \right\|_{\mathcal{Y}_2}+\left\|n\right\|_{\mathcal{Y}_1}\left\|  v \right\|_{\mathcal{Y}_3} \right).
\end{equation}
{\underline{Estimates for $c$}}\newline
From Lemma \ref{heat_kernel2z}, it follows that
\begin{equation}\label{co1}
\left\Vert e^{-\kappa_1 \beta_1 t}e^{t\Delta}c_{0}\right\Vert _{\infty}\leq C \left\Vert c_{0}\right\Vert _{\infty}\!.
\end{equation}
By the assumptions on $p,q,r$, first notice that $\frac{1}{p}+\frac{1}{r}\leq 1,$ $\frac{1}{2}-\frac{N}{2p}>0$ and $\frac{1}{2}-\frac{N}{2}(\frac{1}{q}-\frac{1}{r})>0$. Thus, we can estimate
\begin{eqnarray}\label{co1a}
&&\left\Vert \int_{0}^{t}e^{-\kappa_1 \beta_1(t-\tau)}e^{(t-\tau)\Delta}(u\cdot\nabla c-\kappa_1 \alpha_1 n + \kappa_2 \gamma cn)d\tau\right\Vert _{\infty} \nonumber \\
&&\ \ \ \leq C \int_{0}^{t}(t-\tau)^{-\frac{N}{2p}-\frac{1}{2}}\left\Vert u(\tau)\right\Vert _{p}\left\Vert c(\tau)\right\Vert _{\infty}d\tau \nonumber \\
&&\ \ \ \ \ \ +C \int_{0}^{t}(t-\tau)^{-\frac{N}{2q}}\left\Vert n(\tau)\right\Vert _{q}\left(1+ \left\Vert c(\tau)\right\Vert _{\infty} \right)(\tau)d\tau \nonumber \\
&&\ \ \ \leq C  \left(\sup_{0<\tau<T}\tau^{\frac{N}{2}(\frac{1}{N}-\frac{1}{p})}\left\Vert u(\tau)\right\Vert _{p} \right) \left( \sup_{0<\tau<T}\left\Vert c(\tau)\right\Vert _{\infty} \right) B\left(\frac{1}{2}-\frac{N}{2p},\frac{1}{2}+\frac{N}{2p}\right)\nonumber \\
&&\ \ \ \ \ \ + C \left( \sup_{0<\tau<T}\tau^{\frac{N}{2}(\frac{2}{N}-\frac{1}{q})}\left\Vert n(\tau)\right\Vert _{q}\right)\left(1+\sup_{0<\tau<T}\left\Vert c(\tau)\right\Vert _{\infty} \right)B\left (1-\frac{N}{2q},\frac{N}{2q}\right)\!,
\end{eqnarray}
and
\begin{eqnarray}\label{co2}
&&\left\Vert \nabla\int_{0}^{t}e^{-\kappa_1 \beta_1(t-\tau)}e^{(t-\tau)\Delta}(u\cdot\nabla c-\kappa_1 \alpha_1 n + \kappa_2 \gamma c n)(\tau)d\tau\right\Vert _{r} \nonumber \\
&&\ \ \ \leq C  \int_{0}^{t}(t-\tau)^{-\frac{N}{2p}-\frac{1}{2}}\left\Vert u(\tau)\right\Vert _{p}\left\Vert \nabla c(\tau)\right\Vert _{r}d\tau \nonumber \\
&&\ \ \ \ \ \ +C  \int_{0}^{t}(t-\tau)^{-\frac{N}{2}(\frac{1}{q}-\frac{1}{r})-\frac{1}{2}}\left\Vert n(\tau)\right\Vert _{q}\left(1+ \left\Vert c(\tau)\right\Vert _{\infty} \right)d\tau \nonumber \\
&&\ \ \ \leq C  \left(\sup_{0<\tau<T}\tau^{\frac{N}{2}(\frac{1}{N}-\frac{1}{p})}\left\Vert u(\tau)\right\Vert _{p}\right)\nonumber \\
&&\ \ \ \ \ \ \times \left(\sup_{0<\tau<T}\tau^{\frac{N}{2}(\frac{1}{N}-\frac{1}{r})}\left\Vert \nabla c(\tau)\right\Vert _{q} \right) B\left(\frac{1}{2}-\frac{N}{2p},\frac{N}{2}(\frac{1}{p}+\frac{1}{r})\right)t^{-\frac{N}{2}(\frac{1}{N}-\frac{1}{r})} \nonumber \\
&&\ \ \ \ \ \ +C  \left(\sup_{0<\tau<T}\tau^{\frac{N}{2}(\frac{2}{N}-\frac{1}{q})}\left\Vert n(\tau)\right\Vert _{q} \right)\left(1+\sup_{0<\tau<T}\left\Vert c(\tau)\right\Vert _{\infty} \right)\nonumber \\
&&\ \ \ \ \ \ \times B\left (\frac{1}{2}-\frac{N}{2}(\frac{1}{q}-\frac{1}{r}),\frac{N}{2q}\right)t^{-\frac{N}{2}(\frac{1}{N}-\frac{1}{r})}.
\end{eqnarray}
Using (\ref{co1}), (\ref{co1a}), (\ref{co2}) and taking into account that ${\displaystyle \sup_{0<t<T} t^{\frac{N}{2}(\frac{1}{N}-\frac{1}{r})} \left\Vert \nabla e^{t\Delta}c_{0}\right\Vert _{r}< \infty },$ we obtain

\begin{equation}\label{co4}
\left\|c\right\|_{\mathcal{Y}_2} \leq C  \left( \left\|c_0\right\|_{\mathcal{X}_2}+\left\|c\right\|_{\mathcal{Y}_2}\left\|u\right\|_{\mathcal{Y}_4}+\left\|n\right\|_{\mathcal{Y}_1}\left\|c\right\|_{\mathcal{Y}_2}+\left\|n\right\|_{\mathcal{Y}_1} \right)\!.
\end{equation}

\noindent{\underline{Estimates for $v$}}\newline
Proceeding as in the proof of the estimates for $c$, we also arrive at

\begin{equation}\label{vo4}
\left\|v\right\|_{\mathcal{Y}_3} \leq C \left( \left\|v_0\right\|_{\mathcal{X}_3}+\left\|v\right\|_{\mathcal{Y}_3}\left\|u\right\|_{\mathcal{Y}_4}+\left\|n\right\|_{\mathcal{Y}_1} \right)\!.
\end{equation}

\noindent{\underline{Estimates for $u$}}\newline
First, from Lemma \ref{stokes_kernel} we get
\begin{equation}\label{uo1}
\left\Vert e^{-tA}u_{0}\right\Vert _{p} \leq C  t^{-\frac{N}{2}(\frac{1}{N}-\frac{1}{p})}\left\Vert u_{0}\right\Vert _{N}\!.
\end{equation}
Notice that we have $\frac{1}{s}+\frac{1}{q}\leq 1,$ $1-\frac{N}{2}(\frac{1}{q}+\frac{1}{s}-\frac{1}{p})>0$ and $\frac{N}{2q}+\frac{N}{2s}-\frac{1}{2}>0$. Thus, we can estimate
\begin{eqnarray}
&&\left\Vert \int_{0}^{t}e^{-(t-\tau)A} \mathbb{P} (u\cdot\nabla u+n \nabla\phi)(\tau)d\tau\right\Vert _{p}\nonumber \\
&&\ \ \ \leq\int_{0}^{t}\left\Vert \mathbb{P}  e^{-(t-\tau)A}\nabla\cdot(u\otimes u)(\tau)\right\Vert _{p}d\tau+\int_{0}^{t}\left\Vert \mathbb{P} e^{-(t-\tau)A}(n \nabla\phi )(\tau)\right\Vert _{p}d\tau \nonumber \\
&&\ \ \ \leq C \int_{0}^{t}(t-\tau)^{-\frac{N}{2}(\frac{2}{p}-\frac{1}{p})-\frac{1}{2}}\left\Vert (u\otimes u)(\tau)\right\Vert _{L^{\frac{p}{2}}}d\tau\nonumber \\
&&\ \ \ \ \ \ + C \int_{0}^{t}(t-\tau)^{-\frac{N}{2}(\frac{1}{q}+\frac{1}{s}-\frac{1}{p})}\left\Vert n(\tau)\right\Vert _{q}\left\Vert \nabla\phi (\tau)\right\Vert _{s}d\tau \nonumber \\
&&\ \ \ \leq C  t^{-\frac{N}{2}(\frac{1}{N}-\frac{1}{p})}\left(\sup_{0<\tau<T}\tau^{\frac{N}{2}(\frac{1}{N}-\frac{1}{p})}\left\Vert u( \tau )\right\Vert _{p}\right)^{2}\times B\left(\frac{1}{2}-\frac{N}{2p},\frac{N}{p}\right) \nonumber \\
&&\ \ \ \ \ \ +C  t^{-\frac{N}{2}(\frac{1}{N}-\frac{1}{p})} \left( \sup_{0<\tau<T}\tau^{\frac{1}{2}-\frac{N}{2s}} \left\Vert \nabla\phi ( \tau) \right\Vert _{s} \right)\left(\sup_{0<\tau<T}\tau^{\frac{N}{2}(\frac{2}{N}-\frac{1}{q})}\left\Vert n( \tau )\right\Vert _{q}\right)\nonumber\\
&&\ \ \ \ \ \ \times B\left(1-\frac{N}{2}(\frac{1}{q}+\frac{1}{s}-\frac{1}{p}),\frac{N}{2q}+\frac{N}{2s}-\frac{1}{2}\right)\!.\label{uo2} \nonumber \\
\end{eqnarray}
Recall that in (\ref{uo2}) we can choose an arbitrary representative of $[n(\tau)]$ (see Definition \ref{defimild}). In particular, for a.e. $\tau\in(0,T)$, $n(\tau)$ can be taken as the representative $n^*(\tau)$ that satisfies $\Vert n^*(\tau)\Vert_q=\Vert  [n(\tau)]\Vert_{q/\sim}$. Therefore,
\begin{eqnarray}
\label{uo2a} \sup_{0<\tau<T}\tau^{\frac{N}{2}(\frac{2}{N}-\frac{1}{q})}\left\Vert n( \tau )\right\Vert _{q}=\sup_{0<\tau<T}\tau^{\frac{N}{2}(\frac{2}{N}-\frac{1}{q})}\left\Vert [n( \tau )]\right\Vert _{q/\sim}.
\end{eqnarray}
Hence, from (\ref{uo1}), (\ref{uo2}) and (\ref{uo2a}), it follows that

\begin{equation}\label{uo4}
\left\|u\right\|_{\mathcal{Y}_4} \leq C  \left( \left\|u_0\right\|_{\mathcal{X}_4}+\left\|u\right\|_{\mathcal{Y}_4}\left\|u\right\|_{\mathcal{Y}_4}+\left\|n\right\|_{\mathcal{Y}_1} \right)\!.
\end{equation}
\noindent Using (\ref{no4}), (\ref{co4}), (\ref{vo4}) and (\ref{uo4}), we obtain the following estimates for the vector $[n,c,v,u]$:
\begin{eqnarray} \label{global}
&&\left\|n\right\|_{\mathcal{Y}_1} \leq C e^{\varsigma T} \left( \left\|n_0\right\|_{\mathcal{X}_1}+\left\|n\right\|_{\mathcal{Y}_1}\left\|u\right\|_{\mathcal{Y}_4}+\left\|n\right\|_{\mathcal{Y}_1}\left\|n\right\|_{\mathcal{Y}_1}+\left\|n\right\|_{\mathcal{Y}_1}\left\|  c \right\|_{\mathcal{Y}_2}+\left\|n\right\|_{\mathcal{Y}_1}\left\| v \right\|_{\mathcal{Y}_3} \right)\!, \nonumber \\
&&\left\|c\right\|_{\mathcal{Y}_2} \leq C \left( \left\|c_0\right\|_{\mathcal{X}_2}+\left\|c\right\|_{\mathcal{Y}_2}\left\|u\right\|_{\mathcal{Y}_4}+\left\|n\right\|_{\mathcal{Y}_1}\left\|c\right\|_{\mathcal{Y}_2}+\left\|n\right\|_{\mathcal{Y}_1} \right)\!, \nonumber \\
&&\left\|v\right\|_{\mathcal{Y}_3} \leq C  \left( \left\|v_0\right\|_{\mathcal{X}_3}+\left\|v\right\|_{\mathcal{Y}_3}\left\|u\right\|_{\mathcal{Y}_4}+\left\|n\right\|_{\mathcal{Y}_1} \right)\!, \nonumber \\
&&\left\|u\right\|_{\mathcal{Y}_4} \leq C  \left( \left\|u_0\right\|_{\mathcal{X}_4}+\left\|u\right\|_{\mathcal{Y}_4}\left\|u\right\|_{\mathcal{Y}_4}+\left\|n\right\|_{\mathcal{Y}_1} \right)\!.
\end{eqnarray}
Now, motivated by \cite{choe2017global}, we will consider the following iteration scheme whose limit will provide the global mild solution:
\begin{equation}\label{aproxIE}
\left\{
\begin{array}{lc}
n^{(1)}=e^{\varsigma t}e^{t\Delta}n_{0}, \quad  c^{(1)}=e^{-\kappa_1 \beta_1 t}e^{t\Delta}c_{0}, \quad v^{(1)}=e^{-\beta_2 t}e^{t\Delta}v_{0}, \quad u^{(1)}=e^{-t A}u_{0}, \\[.5cm]
n^{(k+1)}=n^{(1)}-{\displaystyle \int_{0}^{t}e^{\varsigma (t-\tau)}e^{(t-\tau)\Delta}(u^{(k)}\cdot\nabla n^{(k)}+\mu n^{(k)}n^{(k)})(\tau)d\tau}\\[.5cm]
\tab-{\displaystyle \int_{0}^{t} e^{\varsigma (t-\tau)}e^{(t-\tau)\Delta} \left(\nabla\cdot (\chi n^{(k)} \nabla c^{(k)}-\xi n^{(k)} \nabla v^{(k)}) \right) (\tau)d\tau} ,\\[.5cm]
c^{(k+1)}=c^{(1)}-{\displaystyle \int_{0}^{t}e^{-\kappa_1 \beta_1(t-\tau)}e^{(t-\tau)\Delta}(u^{(k)}\cdot\nabla c^{(k)}-\kappa_1 \alpha_1 n^{(k+1)} + \kappa_2 \gamma c^{(k)}n^{(k)})(\tau)d\tau} , \\[.5cm]
v^{(k+1)}=v^{(1)}-{\displaystyle\int_{0}^{t}e^{-\beta_2(t-\tau)}e^{(t-\tau)\Delta}(u^{(k)}\cdot\nabla v^{(k)}-\alpha_2 n^{(k+1)})(\tau)d\tau} , \\[.5cm]
u^{(k+1)}=u^{(1)}-{\displaystyle \int_{0}^{t} e^{-(t-\tau) A} \mathbb{P} (u^{(k)}\cdot\nabla u^{(k)}+n^{(k+1)}\nabla\phi)(\tau)d\tau}.
\end{array}
\right.
\end{equation}
As in (\ref{uo2}), in the last equation of (\ref{aproxIE}) the term $n^{(k+1)}$ is indeed the representative $(n^{(k+1)})^*$ of $[n^{(k+1)}]$. Thus, applying estimates in (\ref{global}) to (\ref{aproxIE}), we obtain the following ones:

\begin{eqnarray} \label{global01}
&&\|n^{(k+1)}\|_{\mathcal{Y}_1} \leq C e^{\varsigma T} \left( \|n_0\|_{\mathcal{X}_1}+\|n^{(k)}\|_{\mathcal{Y}_1}\|u^{(k)}\|_{\mathcal{Y}_4}+\|n^{(k)}\|_{\mathcal{Y}_1}\|n^{(k)}\|_{\mathcal{Y}_1} \right.\nonumber \\
&&\tab[2cm] \left. +\|n^{(k)}\|_{\mathcal{Y}_1}\|  c^{(k)} \|_{\mathcal{Y}_2}+\|n^{(k)}\|_{\mathcal{Y}_1}\|  v^{(k)} \|_{\mathcal{Y}_3} \right)\!, \nonumber \\
&&\|c^{(k+1)}\|_{\mathcal{Y}_2} \leq C \left( \|c_0\|_{\mathcal{X}_2}+\|c^{(k)}\|_{\mathcal{Y}_2}\|u^{(k)}\|_{\mathcal{Y}_4}+\|n^{(k)}\|_{\mathcal{Y}_1}\|c^{(k)}\|_{\mathcal{Y}_2}+\|n^{(k+1)}\|_{\mathcal{Y}_1}\right)\nonumber \\
&&\tab[1.6cm]\leq C \left( \|c_0\|_{\mathcal{X}_2}+\|n_0\|_{\mathcal{X}_1}+\|c^{(k)}\|_{\mathcal{Y}_2}\|u^{(k)}\|_{\mathcal{Y}_4}+\|n^{(k)}\|_{\mathcal{Y}_1}\|c^{(k)}\|_{\mathcal{Y}_2} \right.\nonumber \\
&&\tab[2cm] \left. +\|n^{(k)}\|_{\mathcal{Y}_1}\|u^{(k)}\|_{\mathcal{Y}_4} +\|n^{(k)}\|_{\mathcal{Y}_1}\|n^{(k)}\|_{\mathcal{Y}_1}+\|n^{(k)}\|_{\mathcal{Y}_1}\|  c^{(k)} \|_{\mathcal{Y}_2}+\|n^{(k)}\|_{\mathcal{Y}_1}\|  v^{(k)} \|_{\mathcal{Y}_3} \right)\!, \nonumber \\
&&\|v^{(k+1)}\|_{\mathcal{Y}_3} \leq C \left( \|v_0\|_{\mathcal{X}_3}+\|v^{(k)}\|_{\mathcal{Y}_3}\|u^{(k)}\|_{\mathcal{Y}_4}+ \|n^{(k+1)}\|_{\mathcal{Y}_1} \right) \nonumber \\
&&\tab[1.6cm] \leq C \left( \|v_0\|_{\mathcal{X}_3}+\|n_0\|_{\mathcal{X}_1}+\|v^{(k)}\|_{\mathcal{Y}_3}\|u^{(k)}\|_{\mathcal{Y}_4}+ \|n^{(k)}\|_{\mathcal{Y}_1}\|u^{(k)}\|_{\mathcal{Y}_4} \right. \nonumber \\
&&\tab[2cm] \left. +\|n^{(k)}\|_{\mathcal{Y}_1}\|n^{(k)}\|_{\mathcal{Y}_1} +\|n^{(k)}\|_{\mathcal{Y}_1}\|  c^{(k)} \|_{\mathcal{Y}_2}+\|n^{(k)}\|_{\mathcal{Y}_1}\|  v^{(k)} \|_{\mathcal{Y}_3} \right)\!, \nonumber \\
&&\|u^{(k+1)}\|_{\mathcal{Y}_4} \leq C \left( \|u_0\|_{\mathcal{X}_4}+\|u^{(k)}\|_{\mathcal{Y}_4}\|u^{(k)}\|_{\mathcal{Y}_4} +\|n^{(k+1)}\|_{\mathcal{Y}_1}\right) \nonumber \\
&&\tab[1.6cm] \leq C \left( \|u_0\|_{\mathcal{X}_4}+\|n_0\|_{\mathcal{X}_1}+\|u^{(k)}\|_{\mathcal{Y}_4}\|u^{(k)}\|_{\mathcal{Y}_4} +\|n^{(k)}\|_{\mathcal{Y}_1}\|u^{(k)}\|_{\mathcal{Y}_4}\right.\nonumber \\
&&\tab[2cm] \left. +\|n^{(k)}\|_{\mathcal{Y}_1}\|n^{(k)}\|_{\mathcal{Y}_1} +\|n^{(k)}\|_{\mathcal{Y}_1}\|  c^{(k)} \|_{\mathcal{Y}_2}+\|n^{(k)}\|_{\mathcal{Y}_1}\|  v^{(k)} \|_{\mathcal{Y}_3} \right)\!.
\end{eqnarray}
For small initial data, the sequence $\left[n^{(k)}, c^{(k)}, v^{(k)}, u^{(k)} \right]$ is uniformly bounded in the space $\mathcal{Y}.$ In fact, suppose that
\begin{equation}\label{bola}
\left\|[n^{(k)}, c^{(k)}, v^{(k)}, u^{(k)}] \right\|_{\mathcal{Y}} \leq R.
\end{equation}
Then from (\ref{global01}), it is straightforward to get

\begin{equation*}\label{bola01}
\left\|[n^{(k+1)}, c^{(k+1)}, v^{(k+1)}, u^{(k+1)}] \right\|_{\mathcal{Y}} \leq C(X_0 + 20 R^2),
\end{equation*}
where $X_0=4\|n_0\|_{\mathcal{X}_1}+\|c_0\|_{\mathcal{X}_2}+\|v_0\|_{\mathcal{X}_3}+\|u_0\|_{\mathcal{X}_4}$ and $C$ is a positive constant. Then, for $X_0$ small enough, we can consider the smallest number $R$, namely
\begin{equation*}\label{radio}
R=\frac{1-\sqrt{1-80X_0C^2}}{40C}>0,
\end{equation*}
such that
\begin{equation*}\label{radio01}
C(X_0 + 20 R^2)=R.
\end{equation*}
Thus, the sequence $\left[n^{(k)}, c^{(k)}, v^{(k)}, u^{(k)} \right],\ k\in\mathbb{N},$ is uniformly bounded in $\mathcal{Y}$. Next, let us consider the difference
$$\left[n^{(k+1)}-n^{(k)}, c^{(k+1)}-c^{(k)}, v^{(k+1)}-v^{(k)}, u^{(k+1)}-u^{(k)} \right]\!.$$
For the first component $n^{(k+1)}-n^{(k)},$ we have

\begin{equation*}\label{restoIEn}
\left.
\begin{array}{lc}
n^{(k+1)}-n^{(k)}={\displaystyle \int_{0}^{t}e^{\varsigma (t-\tau)}e^{(t-\tau)\Delta} \left(u^{(k-1)}\cdot\nabla n^{(k-1)}-u^{(k)}\cdot\nabla n^{(k)} \right) (\tau) d\tau}\\[.5cm]
\tab[2.4cm] +{\displaystyle \int_{0}^{t}e^{\varsigma (t-\tau)}e^{(t-\tau)\Delta} \left( \mu n^{(k-1)}n^{(k-1)}-\mu n^{(k)}n^{(k)}\right)(\tau)d\tau}\\[.5cm]
\tab[2.4cm] +{\displaystyle \int_{0}^{t} e^{\varsigma (t-\tau)}e^{(t-\tau)\Delta} \left( \nabla\cdot (\chi n^{(k-1)} \nabla c^{(k-1)}-\chi n^{(k)} \nabla c^{(k)}) \right)(\tau)d\tau} \\[.5cm]
\tab[2.4cm] - {\displaystyle \int_{0}^{t} e^{\varsigma (t-\tau)}e^{(t-\tau)\Delta}\nabla\cdot \left(\xi n^{(k-1)} \nabla v^{(k-1)}-\xi n^{(k)} \nabla v^{(k)} \right)(\tau)d\tau}.
\end{array}
\right.
\end{equation*}
Using (\ref{no2}) and (\ref{no3}), we get
\begin{eqnarray}
&&\| n^{(k+1)}-n^{(k)}\|_{\mathcal{Y}_1} \leq C \left( \|n^{(k-1)}\|_{\mathcal{Y}_1}\|u^{(k-1)}-u^{(k)}\|_{\mathcal{Y}_4}+\|n^{(k-1)}-n^{(k)}\|_{\mathcal{Y}_1}\|u^{(k)}\|_{\mathcal{Y}_4} \right. \nonumber\\
&&\tab[2.9cm] \ \ \ +\|n^{(k-1)}\|_{\mathcal{Y}_1}\|n^{(k-1)}-n^{(k)}\|_{\mathcal{Y}_1} +\|n^{(k-1)}-n^{(k)}\|_{\mathcal{Y}_1}\|n^{(k)}\|_{\mathcal{Y}_1}\nonumber\\ [.3cm]
&&\tab[2.9cm] \ \ \ +\|n^{(k-1)}\|_{\mathcal{Y}_1}\|  c^{(k-1)} - c^{(k)} \|_{\mathcal{Y}_2} +\|n^{(k-1)}-n^{(k)}\|_{\mathcal{Y}_1}\|  c^{(k)} \|_{\mathcal{Y}_2}\nonumber \\ [.3cm]
&&\tab[2.9cm]\ \ \  \left. +\|n^{(k-1)}\|_{\mathcal{Y}_1}\|  v^{(k-1)} - v^{(k)} \|_{\mathcal{Y}_3} +\|n^{(k-1)}-n^{(k)}\|_{\mathcal{Y}_1}\|  v^{(k)} \|_{\mathcal{Y}_3} \right)\nonumber\\ [.2cm]
&&\tab[2.4cm] \ \ \ \leq CR \|[n^{(k-1)}-n^{(k)} \text{,}\ c^{(k)}-c^{(k-1)} \text{,}\ v^{(k-1)}-v^{(k)} \text{,}\ u^{(k-1)}-u^{(k)}] \|_{\mathcal{Y}}.
\label{resto0n04}
\end{eqnarray}
For the component $c^{(k+1)}-c^{(k)},$ we have

\begin{equation*}\label{restoIEc}
\left.
\begin{array}{lc}
c^{(k+1)}-c^{(k)}={\displaystyle \int_{0}^{t}e^{-\kappa_1 \beta_1(t-\tau)}e^{(t-\tau)\Delta}\left(u^{(k-1)}\cdot\nabla c^{(k-1)}-u^{(k)}\cdot\nabla c^{(k)}\right) (\tau) d\tau}\\[.5cm]
\tab[2.4cm] +{\displaystyle \int_{0}^{t}e^{-\kappa_1 \beta_1(t-\tau)}e^{(t-\tau)\Delta}\kappa_2 \gamma \left(  c^{(k-1)}n^{(k-1)}-c^{(k)}n^{(k)} \right)(\tau)d\tau}\\[.5cm]
\tab[2.4cm] +{\displaystyle \int_{0}^{t} e^{-\kappa_1 \beta_1(t-\tau)}e^{(t-\tau)\Delta} \kappa_1 \alpha_1 \left( n^{(k+1)}-n^{(k)} \right)(\tau)d\tau}. \\.
\end{array}
\right.
\end{equation*}
Using (\ref{co1a}), (\ref{co2}) and (\ref{resto0n04}), we get
\begin{eqnarray}\label{resto0c04}
&&\|c^{(k+1)}-c^{(k)}\|_{\mathcal{Y}_2} \leq C \left( \|c^{(k-1)}\|_{\mathcal{Y}_2}\|u^{(k-1)}-u^{(k)}\|_{\mathcal{Y}_4}+\|c^{(k-1)}-c^{(k)}\|_{\mathcal{Y}_2}\|u^{(k)}\|_{\mathcal{Y}_4} \right.\nonumber\\[.2cm]
&&\tab[2.9cm] \ \ \ \left. +\|n^{(k-1)}\|_{\mathcal{Y}_1}\|c^{(k-1)}-c^{(k)}\|_{\mathcal{Y}_2} +\|n^{(k-1)}-n^{(k)}\|_{\mathcal{Y}_1}\|c^{(k)}\|_{\mathcal{Y}_2} \right.\nonumber\\[.2cm]
&&\tab[2.9cm] \ \ \ \left.+\|n^{(k+1)}-n^{(k)}\|_{\mathcal{Y}_1} \right)\nonumber\\[.1cm]
&&\tab[2.5cm] \ \ \ \leq CR \| [n^{(k-1)}-n^{(k)} \text{,}\ c^{(k-1)}-c^{(k)} \text{,}\ v^{(k-1)}-v^{(k)} \text{,}\ u^{(k-1)}-u^{(k)}] \|_{\mathcal{Y}}\!.
\end{eqnarray}
Similarly, we arrive at

\begin{equation}
\left.
\begin{array}{lc}
\left\|v^{(k+1)}-v^{(k)}\right\|_{\mathcal{Y}_3} \leq CR \left\| [ n^{(k-1)}-n^{(k)} \text{,}\ c^{(k-1)}-c^{(k)} \text{,}\ v^{(k-1)}-v^{(k)} \text{,}\ u^{(k-1)}-u^{(k)} ] \right\|_{\mathcal{Y}}\!,\\[.5cm]
\left\|u^{(k+1)}-u^{(k)}\right\|_{\mathcal{Y}_4} \leq CR \left\|[ n^{(k-1)}-n^{(k)} \text{,}\ c^{(k-1)}-c^{(k)} \text{,}\ v^{(k-1)}-v^{(k)} \text{,}\ u^{(k-1)}-u^{(k)} ] \right\|_{\mathcal{Y}}\!.\label{resto02}
\end{array}
\right.
\end{equation}
Combining the estimates (\ref{resto0n04}), (\ref{resto0c04}) and (\ref{resto02}), the result is

\begin{equation}\label{resto03}
\left.
\begin{array}{lc}
\left\|[n^{(k+1)}-n^{(k)} \text{,}\ c^{(k+1)}-c^{(k)} \text{,}\ v^{(k+1)}-v^{(k)} \text{,}\ u^{(k+1)}-u^{(k)}] \right\|_{\mathcal{Y}}  \\[.5cm]
\leq 4CR \left\|[n^{(k-1)}-n^{(k)} \text{,}\ c^{(k-1)}-c^{(k)} \text{,}\ v^{(k-1)}-v^{(k)} \text{,}\ u^{(k-1)}-u^{(k)}] \right\|_{\mathcal{Y}}\!.
\end{array}
\right.
\end{equation}
Reducing $X_0$ (if necessary), we can take $R$ such that $R < \frac{1}{4C}$, and then the sequence $\left[n^{(k)}, c^{(k)}, v^{(k)}, u^{(k)} \right]\!,$ $k\in\mathbb{N},$
is a Cauchy sequence in $\mathcal{Y}$. Thus, the limit $\left[n, c, v, u \right]$ solves the equations (\ref{mapIE}) in $\mathcal{Y}$. Finally, we observe that estimates (\ref{bola}) and (\ref{resto03}), with slight modifications to consider two possible solutions $[n, c, v, u]$ and $[\tilde{n}, \tilde{c}, \tilde{v}, \tilde{u}]$, also assure the uniqueness of solutions in the closed ball $\{[n, c, v, u]\in\mathcal{Y};\left\|[n, c, v, u]\right\|_{\mathcal{Y}} \leq R\}$.
\end{proof}
\begin{remark}\label{rem1a}
The range of $q$ determines the conditions on the exponents $p$ and $r.$ The difference in the case $N=2$ and $N=3$ in the proof of Theorem \ref{theo1} is clarified in the following comments.
\begin{enumerate}
\item Notice that  if $s=N,$ the condition $\frac{1}{s}+\frac{1}{q}\leq 1$ to obtain (\ref{uo2}) implies that $N\neq 2.$ In fact, if $N=2$, then $q \geq 2$ which is incompatible with the condition $\frac{N}{q}-1>0$ given in (\ref{no2}).
\item In order to get (\ref{no2}) and (\ref{uo2}) we need $q>\frac{N}{2}$ and $q\geq \frac{s}{s-1},$ respectively. If $N=3,$ then $\frac{s}{s-1}\leq \frac{N}{2},$ and therefore we assume $q>\frac{N}{2}$ in condition $\textit{(}i\textit{)}$ in Theorem \ref{theo1}. On the other hand, if $N=2,$ then $\frac{N}{2}<\frac{s}{s-1},$ and therefore we assume $q\geq\frac{s}{s-1}$ in condition $\textit{(}ii\textit{)}$ in Theorem \ref{theo1}.
\item If $N=3,$ the condition $1>\frac{N}{2}(\frac{1}{q}+\frac{1}{s}-\frac{1}{p})$ (equivalently, $p<\frac{Nqs}{Ns+Nq-2sq}$) is necessary to obtain (\ref{uo2}). In the case $N=2,$ it is trivially satisfied for $p<\infty$ since $\frac{1}{s}+\frac{1}{q}\leq 1.$
\item For (\ref{no2}) and (\ref{co1a}) we need to assume $p>N$. For (\ref{no2}) we also need $\frac{q}{q-1}\leq p$. Taking into account the range for $q$, if $N=3$ we have $\frac{q}{q-1}<N$. Therefore, we assume $N<p$ in the condition $\textit{(}i\textit{)}$ of Theorem \ref{theo1}. If $N=2$ we have $N<\frac{q}{q-1}$ and thus we assume $\frac{q}{q-1}\leq p$ in the conditions $\textit{(}ii\textit{)}$ and $\textit{(}iii\textit{)}$ of Theorem \ref{theo1}.
\end{enumerate}
\end{remark}

\subsection{Proof of Theorem  \ref{theo2}}
In the proof of Theorem \ref{theo2} we will need a lemma related to the integrability of a beta type function. Although it seems to be more or less known, we have not been able to locate the exact statement and its proof in the literature. So, for reader's convenience, we include them here.
\begin{lemm} \label{betamod}
Let $x<1,$ $y<1$ and $a,b>0$. Then

\begin{equation}\label{betaz}
\int_{0}^t(t-\tau)^{-x} \tau^{-y}e^{-a(t-\tau)}e^{-b\tau} d\tau \leq e^{-\min{\{a,b\}}}t^{1-x-y}B(1-x,1-y).
\end{equation}
\end{lemm}
\begin{proof}
First assume that $b\geq a.$ Then we have
\begin{eqnarray*}
\int_{0}^t(t-\tau)^{-x} \tau^{-y}e^{-a(t-\tau)}e^{-b\tau} d\tau &= & e^{-at}\int_{0}^t(t-\tau)^{-x} \tau^{-y}e^{-(b-a)\tau} d\tau \\
&\leq & e^{-at}\int_{0}^t(t-\tau)^{-x} \tau^{-y} d\tau \\
&\leq & e^{-at}t^{1-x-y}B(1-x,1-y).
\end{eqnarray*}
On the other hand if $b<a$, making the change of variable $\tau=t-\check{\tau},$ we get
\begin{eqnarray*}
\int_{0}^t(t-\tau)^{-x} \tau^{-y}e^{-a(t-\tau)}e^{-b\tau} d\tau &= & \int_{0}^t(\check{\tau})^{-x} (t-\check{\tau})^{-y}e^{-a(\check{\tau})}e^{-b(t-\check{\tau})} d\check{\tau} \\
&\leq & e^{-bt}t^{1-x-y}B(1-y,1-x).
\end{eqnarray*}
Recalling that the beta function is symmetric, we get (\ref{betaz}).
\end{proof}
\begin{proof}[Proof of Theorem \ref{theo2}]

If in (\ref{KNS}) we change $\varsigma n-\mu n^2$ by $-\tilde{\varsigma} n$, then the term $e^{\varsigma t}$ in (\ref{mapIE}) is replaced by $e^{-\tilde{\varsigma} t}$. Thus, arguing in a similar way to the proof of Theorem \ref{theo1} and using Lemma \ref{betamod}, we can obtain the estimates below for $n$, $c$ and $u$.\\
{\underline{Estimates for $n$}}\newline
From Lemma \ref{heat_kernel2z}, we get

\begin{equation}\label{no1z}
\left\Vert e^{-\tilde{\varsigma} t} e^{t \Delta }n_{0}\right\Vert _{q}\leq Ce^{-\tilde{\varsigma} t} t^{-\frac{N}{2}(\frac{2}{N}-\frac{1}{q})}\left\Vert n_{0}\right\Vert _{\frac{N}{2}}\!.
\end{equation}
Since $\frac{1}{p}+\frac{1}{q}\leq 1$, $\frac{1}{2}-\frac{N}{2p}>0$ and $\frac{N}{2}(\frac{1}{p}+\frac{1}{q})-\frac{1}{2}>0$, we can employ Lemmas \ref{heat_kernel2} and \ref{betamod} in order to estimate

\begin{eqnarray}\label{no2z}
&&\left\Vert \int_{0}^{t}e^{-\tilde{\varsigma} (t-\tau)}e^{(t-\tau)\Delta}(u\cdot\nabla n)(\tau)d\tau\right\Vert _{q}=\left\Vert \int_{0}^{t}e^{-\tilde{\varsigma} (t-\tau)}e^{(t-\tau)\Delta}(\nabla\cdot(n u))(\tau)d\tau\right\Vert _{q}\nonumber\\
&&\ \ \ \leq C\int_{0}^{t}e^{-\tilde{\varsigma} (t-\tau)}(t-\tau)^{-\frac{N}{2p}-\frac{1}{2}}\left\Vert n(\tau)\right\Vert _{q}\left\Vert u(\tau)\right\Vert _{p}d\tau \nonumber \\
&&\ \ \ \leq C\int_{0}^{t}e^{-\tilde{\varsigma} (t-\tau)}e^{-(\tilde{\varsigma}+\rho_2)\tau}(t-\tau)^{-\frac{N}{2p}-\frac{1}{2}}\tau^{\frac{N}{2}(\frac{1}{p}+\frac{1}{q})-\frac{3}{2}} \tau^{\frac{N}{2}(\frac{2}{N}-\frac{1}{q})} e^{\tilde{\varsigma} \tau} \left\Vert n(\tau)\right\Vert _{q} \tau^{\frac{N}{2}(\frac{1}{N}-\frac{1}{p})}e^{\rho_2 t} \left\Vert u(\tau)\right\Vert _{p}d\tau \nonumber \\
&&\ \ \ \leq C e^{-\tilde{\varsigma} t} t^{-\frac{N}{2}(\frac{2}{N}-\frac{1}{q})} \left(\sup_{0<\tau<\infty}\tau^{\frac{N}{2}(\frac{2}{N}-\frac{1}{q})} e^{\tilde{\varsigma} \tau} \left\Vert n( \tau )\right\Vert _{q} \right) \left( \sup_{0<\tau<\infty}\tau^{\frac{N}{2}(\frac{1}{N}-\frac{1}{p})}e^{\rho_2 t} \left\Vert u( \tau )\right\Vert _{p} \right) \nonumber \\
&&\ \ \ \ \ \times B\left (\frac{1}{2}-\frac{N}{2p},\frac{N}{2}(\frac{1}{p}+\frac{1}{q})-\frac{1}{2}\right)\!, \nonumber \\
\end{eqnarray}
\noindent where $C=C(\Omega,p,q)>0$.

Next, notice that we have the following relations between $r$ and $q$: $\frac{1}{q}+\frac{1}{r}\leq 1$, $\frac{1}{2}-\frac{N}{2r}>0$ and $\frac{N}{2}(\frac{1}{q}+\frac{1}{r}-\frac{1}{N})>0$. Thus, using again Lemmas \ref{heat_kernel2} and \ref{betamod}, we obtain

\begin{eqnarray}\label{no3z}
&&\left\Vert \int_{0}^{t}e^{-\tilde{\varsigma}(t-\tau)}e^{(t-\tau)\Delta} \left(\vphantom{n^{(k)}} \nabla\cdot(\chi n\nabla c-\xi n\nabla v) \right) (\tau)d\tau\right\Vert _{q}\nonumber \\
&&\ \ \  \leq C\int_{0}^{t}e^{-\tilde{\varsigma}(t-\tau)}(t-\tau)^{-\frac{N}{2r}-\frac{1}{2}}\left\Vert n(\tau)\right\Vert _{q}(\left\Vert \nabla c(\tau)\right\Vert _{r}+\left\Vert \nabla v(\tau)\right\Vert _{r})d\tau \nonumber \\
&&\ \ \ \leq C e^{-\tilde{\varsigma} t} t^{-\frac{N}{2}(\frac{2}{N}-\frac{1}{q})} \left( \sup_{0<\tau<\infty}\tau^{\frac{N}{2}(\frac{2}{N}-\frac{1}{q})} e^{\tilde{\varsigma} \tau} \left\Vert n(\tau)\right\Vert _{q} \right) \nonumber \\
&&\ \ \ \ \ \ \times \left( \sup_{0<\tau<\infty}\tau^{\frac{N}{2}(\frac{1}{N}-\frac{1}{r})} \left\Vert \nabla c(\tau)\right\Vert _{r}+\sup_{0<\tau<\infty}\tau^{\frac{N}{2}(\frac{1}{N}-\frac{1}{r})} \left\Vert \nabla v(\tau)\right\Vert _{r}\right)\nonumber \\
&&\ \ \ \ \ \ \times B\left(\frac{1}{2}-\frac{N}{2r},\frac{N}{2}(\frac{1}{q}+\frac{1}{r}-\frac{1}{N})\right)\!.
\end{eqnarray}

Putting together the estimates (\ref{no1z}) to (\ref{no3z}), we get

\begin{equation}\label{no4z}
\left\|n\right\|_{\mathcal{Y}^{\tiny{\mbox{exp}}}_1} \leq \left( \left\|n_0\right\|_{\mathcal{X}_1}+\left\|n\right\|_{\mathcal{Y}^{\tiny{\mbox{exp}}}_1}\left\|u\right\|_{\mathcal{Y}^{\tiny{\mbox{exp}}}_4}+\left\|n\right\|_{\mathcal{Y}^{\tiny{\mbox{exp}}}_1}\left\| c \right\|_{\mathcal{Y}^{\tiny{\mbox{exp}}}_2}+\left\|n\right\|_{\mathcal{Y}^{\tiny{\mbox{exp}}}_1}\left\|  v \right\|_{\mathcal{Y}^{\tiny{\mbox{exp}}}_3} \right).
\end{equation}
{\underline{Estimates for $c$}}\newline
Firstly, Lemma \ref{heat_kernel2z} yields
\begin{equation}\label{co1z}
\left\Vert e^{-\kappa_1 \beta_1 t} e^{t\Delta}c_{0}\right\Vert _{\infty}\leq C e^{-\kappa_1 \beta_1 t}\left\Vert c_{0}\right\Vert _{\infty}\!.
\end{equation}
The assumptions on $p,q,r$ imply that $1-\frac{N}{2q},$ $\frac{1}{p}+\frac{1}{r}\leq 1,$ $\frac{1}{2}-\frac{N}{2p}>0$ and $\frac{1}{2}-\frac{N}{2}(\frac{1}{q}-\frac{1}{r})>0$. Then, using Lemmas \ref{heat_kernel2}, \ref{heat_kernel2z} and \ref{betamod}, we can estimate
\begin{eqnarray}\label{co1az}
&&\left\Vert \int_{0}^{t}e^{-\kappa_1 \beta_1(t-\tau)}e^{(t-\tau)\Delta}(u\cdot\nabla c-\kappa_1 \alpha_1 n + \kappa_2 \gamma c n)(\tau)d\tau\right\Vert _{\infty} \nonumber \\
&&\ \ \ \leq C \int_{0}^{t}e^{-\kappa_1 \beta_1(t-\tau)}(t-\tau)^{-\frac{N}{2p}-\frac{1}{2}}\left\Vert u(\tau)\right\Vert _{p}\left\Vert c(\tau)\right\Vert _{\infty} d\tau \nonumber \\
&&\ \ \ \ \ \ +C \int_{0}^{t}e^{-\kappa_1 \beta_1(t-\tau)}(t-\tau)^{-\frac{N}{2q}}\left\Vert n(\tau)\right\Vert _{q}\left(1+ \left\Vert c(\tau)\right\Vert _{\infty} \right) d\tau \nonumber \\
&&\ \ \ \leq C  e^{-\kappa_1 \beta_1 t} \left( \sup_{0<\tau<\infty}\tau^{\frac{N}{2}(\frac{1}{N}-\frac{1}{p})} e^{\rho_2 \tau }\left\Vert u(\tau)\right\Vert _{p} \right) \nonumber \\
&&\ \ \ \ \ \ \times \left(\sup_{0<\tau<\infty} e^{\kappa_1 \beta_1 \tau }\left\Vert c(\tau)\right\Vert _{\infty} \right) B\left(\frac{1}{2}-\frac{N}{2p},\frac{1}{2}+\frac{N}{2p}\right)\nonumber \\
&&\ \ \ \ \ \ + C e^{-\min \{\kappa_1 \beta_1,\tilde{\varsigma} \} t} \left( \sup_{0<\tau<\infty}\tau^{\frac{N}{2}(\frac{2}{N}-\frac{1}{q})} e^{\tilde{\varsigma} \tau } \left\Vert n(\tau)\right\Vert _{q} \right)B\left (1-\frac{N}{2q},\frac{N}{2q}\right) \\
&&\ \ \ \ \ \ + C e^{-\kappa_1 \beta_1 t} \left( \sup_{0<\tau<\infty}\tau^{\frac{N}{2}(\frac{2}{N}-\frac{1}{q})}e^{\tilde{\varsigma} \tau }\left\Vert n(\tau)\right\Vert _{q} \right) \left( \sup_{0<\tau<\infty} e^{\kappa_1 \beta_1 \tau}\left\Vert  c(\tau)\right\Vert _{\infty} \right) B\left (1-\frac{N}{2q},\frac{N}{2q}\right)\!, \nonumber
\end{eqnarray}
and
\begin{eqnarray}\label{co2z}
&&\left\Vert \nabla\int_{0}^{t}e^{-\kappa_1 \beta_1(t-\tau)}e^{(t-\tau)\Delta}(u\cdot\nabla c-\kappa_1 \alpha_1 n + \kappa_2 \gamma c n)(\tau)d\tau\right\Vert _{r} \nonumber \\
&&\ \ \ \leq C  \int_{0}^{t}e^{-\kappa_1 \beta_1(t-\tau)}(t-\tau)^{-\frac{N}{2p}-\frac{1}{2}}\left\Vert u(\tau)\right\Vert _{p}\left\Vert \nabla c(\tau)\right\Vert _{r}d\tau \nonumber \\
&&\ \ \ \ \ \ +C  \int_{0}^{t}e^{-\kappa_1 \beta_1(t-\tau)}(t-\tau)^{-\frac{N}{2}(\frac{1}{q}-\frac{1}{r})-\frac{1}{2}}\left\Vert n(\tau)\right\Vert _{q}\left(1+ \left\Vert c(\tau)\right\Vert _{\infty} \right)d\tau \nonumber \\
&&\ \ \ \leq C e^{-\min \{\kappa_1 \beta_1,\rho_2 \} t} t^{-\frac{N}{2}(\frac{1}{N}-\frac{1}{r})} \left(\sup_{0<\tau<\infty}\tau^{\frac{N}{2}(\frac{1}{N}-\frac{1}{p})} e^{\rho_2 \tau}\left\Vert u(\tau)\right\Vert _{p}\right)\nonumber \\
&&\ \ \ \ \ \ \times \left( \sup_{0<\tau<\infty}\tau^{\frac{N}{2}(\frac{1}{N}-\frac{1}{r})} \left\Vert \nabla c(\tau)\right\Vert _{q} \right) B\left(\frac{1}{2}-\frac{N}{2p},\frac{N}{2}(\frac{1}{p}+\frac{1}{r})\right) \nonumber \\
&&\ \ \ \ \ \ +C e^{-\min \{\kappa_1 \beta_1,\tilde{\varsigma} \} t} t^{-\frac{N}{2}(\frac{1}{N}-\frac{1}{r})} \left( \sup_{0<\tau<\infty}\tau^{\frac{N}{2}(\frac{2}{N}-\frac{1}{q})} e^{\tilde{\varsigma} \tau} \left\Vert n(\tau)\right\Vert _{q} \right) B\left (\frac{1}{2}-\frac{N}{2}(\frac{1}{q}-\frac{1}{r}),\frac{N}{2q}\right) \nonumber \\
&&\ \ \ \ \ \ +C e^{-\kappa_1 \beta_1 t}\left( \sup_{0<\tau<\infty}\tau^{\frac{N}{2}(\frac{2}{N}-\frac{1}{q})} e^{\tilde{\varsigma} \tau} \left\Vert n(\tau)\right\Vert _{q}\right) \left(\sup_{0<\tau<\infty} e^{\kappa_1 \beta_1 \tau}\left\Vert c(\tau)\right\Vert _{\infty} \right) \nonumber \\
&&\ \ \ \ \ \ \times B\left (\frac{1}{2}-\frac{N}{2}(\frac{1}{q}-\frac{1}{r}),\frac{N}{2q}\right)t^{-\frac{N}{2}(\frac{1}{N}-\frac{1}{r})}. \nonumber \\
\end{eqnarray}
 In view of (\ref{co1z}), (\ref{co1az}) and (\ref{co2z}), and recalling that ${\displaystyle \sup_{0<t<\infty} t^{\frac{N}{2}(\frac{1}{N}-\frac{1}{r})} \left\Vert \nabla e^{t\Delta}c_{0}\right\Vert _{r}< \infty }$, we obtain

\begin{equation}\label{co4z}
\left\|c\right\|_{\mathcal{Y}^{\tiny{\mbox{exp}}}_2} \leq C \left( \left\|c_0\right\|_{\mathcal{X}_2}+\left\|c\right\|_{\mathcal{Y}^{\tiny{\mbox{exp}}}_2}\left\|u\right\|_{\mathcal{Y}^{\tiny{\mbox{exp}}}_4}+\left\|n\right\|_{\mathcal{Y}^{\tiny{\mbox{exp}}}_1}\left\|c\right\|_{\mathcal{Y}^{\tiny{\mbox{exp}}}_2}+\left\|n\right\|_{\mathcal{Y}^{\tiny{\mbox{exp}}}_1} \right)\!.
\end{equation}

\noindent{\underline{Estimates for $v$}}\newline
Proceeding as above, we also can obtain the estimate below for $v$. The proof is left to the reader.

\begin{equation}\label{vo4z}
\left\|v\right\|_{\mathcal{Y}^{\tiny{\mbox{exp}}}_3} \leq C \left( \left\|v_0\right\|_{\mathcal{X}_3}+\left\|v\right\|_{\mathcal{Y}^{\tiny{\mbox{exp}}}_3}\left\|u\right\|_{\mathcal{Y}^{\tiny{\mbox{exp}}}_4}+\left\|n\right\|_{\mathcal{Y}^{\tiny{\mbox{exp}}}_1} \right)\!.
\end{equation}

\noindent{\underline{Estimates for $u$}}\newline
First, from Lemma \ref{stokes_kernel} we get
\begin{equation}\label{uo1z}
\left\Vert e^{-tA}u_{0}\right\Vert _{p} \leq C e^{-\rho_2 t} t^{-\frac{N}{2}(\frac{1}{N}-\frac{1}{p})}\left\Vert u_{0}\right\Vert _{N}\!.
\end{equation}
From assumptions on $p$ and $q$, we have $\frac{1}{N}+\frac{1}{q}\leq 1$, $\frac{1}{2}-\frac{N}{2p}>0$ and $\frac{1}{2}-\frac{N}{2}(\frac{1}{q}-\frac{1}{p})>0$; therefore, we can estimate
\begin{eqnarray}
&&\left\Vert \int_{0}^{t}e^{-(t-\tau)A} \mathbb{P} (u\cdot\nabla u+n \nabla\phi)(\tau)d\tau\right\Vert _{p}\nonumber \\
&&\ \ \ \leq C  \int_{0}^{t}e^{-\rho_2 (t-\tau)}(t-\tau)^{-\frac{N}{2p}-\frac{1}{2}}\left\Vert u(\tau)\right\Vert _{p}\left\Vert u(\tau)\right\Vert _{p}d\tau \nonumber \\
&&\ \ \ \ \ \ +C  \int_{0}^{t}e^{-\rho_2 (t-\tau)}(t-\tau)^{-\frac{N}{2}(\frac{1}{q}-\frac{1}{p})-\frac{1}{2}}\left\Vert n(\tau)\right\Vert _{q}\left\Vert \nabla\phi (\tau)\right\Vert _{N}d\tau \nonumber \\
&&\ \ \ \leq C e^{-\rho_2 t} t^{-\frac{N}{2}(\frac{1}{N}-\frac{1}{p})}\left(\sup_{0<\tau<\infty} \tau^{\frac{N}{2}(\frac{1}{N}-\frac{1}{p})} e^{\rho_2 \tau} \left\Vert u( \tau)\right\Vert _{p}\right)^{2}\times B\left(\frac{1}{2}-\frac{N}{2p},\frac{N}{p}\right) \nonumber \\
&&\ \ \ \ \ \ +C e^{-\min \{ \rho_2,\tilde{\varsigma} \} t} t^{-\frac{N}{2}(\frac{1}{N}-\frac{1}{p})} \left( \sup_{0<\tau<\infty} \left\Vert \nabla\phi ( \tau )\right\Vert _{N}\right) \left( \sup_{0<\tau<\infty}\tau^{\frac{N}{2}(\frac{2}{N}-\frac{1}{q})} e^{\tilde{\varsigma} \tau } \left\Vert n( \tau)\right\Vert _{q}\right)\nonumber\\
&&\ \ \ \ \ \ \times B\left(\frac{1}{2}-\frac{N}{2}(\frac{1}{q}-\frac{1}{p}),\frac{N}{2q}\right)\!.\label{uo2z}
\end{eqnarray}
Estimates (\ref{uo1z}) and (\ref{uo2z}) together yield

\begin{equation}\label{uo4z}
\left\|u\right\|_{\mathcal{Y}^{\tiny{\mbox{exp}}}_4} \leq C e^{-\min\{\rho_2,\tilde{\varsigma}\} t} \left( \left\|u_0\right\|_{\mathcal{X}_4}+\left\|u\right\|_{\mathcal{Y}^{\tiny{\mbox{exp}}}_4}\left\|u\right\|_{\mathcal{Y}^{\tiny{\mbox{exp}}}_4}+\left\|n\right\|_{\mathcal{Y}^{\tiny{\mbox{exp}}}_1} \right)\!.
\end{equation}

\noindent Hence, from (\ref{no4z}), (\ref{co4z}), (\ref{vo4z}) and (\ref{uo4z}), we obtain the estimates
\begin{eqnarray} \label{globalz}
&&\left\|n\right\|_{\mathcal{Y}^{\tiny{\mbox{exp}}}_1} \leq C \left( \left\|n_0\right\|_{\mathcal{X}_1}+\left\|n\right\|_{\mathcal{Y}^{\tiny{\mbox{exp}}}_1}\left\|u\right\|_{\mathcal{Y}^{\tiny{\mbox{exp}}}_4}+\left\|n\right\|_{\mathcal{Y}^{\tiny{\mbox{exp}}}_1}\left\|  c \right\|_{\mathcal{Y}^{\tiny{\mbox{exp}}}_2}+\left\|n\right\|_{\mathcal{Y}^{\tiny{\mbox{exp}}}_1}\left\| v \right\|_{\mathcal{Y}^{\tiny{\mbox{exp}}}_3} \right)\!, \nonumber \\
&&\left\|c\right\|_{\mathcal{Y}^{\tiny{\mbox{exp}}}_2} \leq C \left( \left\|c_0\right\|_{\mathcal{X}_2}+\left\|c\right\|_{\mathcal{Y}^{\tiny{\mbox{exp}}}_2}\left\|u\right\|_{\mathcal{Y}^{\tiny{\mbox{exp}}}_4}+\left\|n\right\|_{\mathcal{Y}^{\tiny{\mbox{exp}}}_1}\left\|c\right\|_{\mathcal{Y}^{\tiny{\mbox{exp}}}_2}+\left\|n\right\|_{\mathcal{Y}^{\tiny{\mbox{exp}}}_1} \right)\!, \nonumber \\
&&\left\|v\right\|_{\mathcal{Y}^{\tiny{\mbox{exp}}}_3} \leq C \left( \left\|v_0\right\|_{\mathcal{X}_3}+\left\|v\right\|_{\mathcal{Y}^{\tiny{\mbox{exp}}}_3}\left\|u\right\|_{\mathcal{Y}^{\tiny{\mbox{exp}}}_4}+\left\|n\right\|_{\mathcal{Y}^{\tiny{\mbox{exp}}}_1} \right)\!, \nonumber \\
&&\left\|u\right\|_{\mathcal{Y}^{\tiny{\mbox{exp}}}_4} \leq C \left( \left\|u_0\right\|_{\mathcal{X}_4}+\left\|u\right\|_{\mathcal{Y}^{\tiny{\mbox{exp}}}_4}\left\|u\right\|_{\mathcal{Y}^{\tiny{\mbox{exp}}}_4}+\left\|n\right\|_{\mathcal{Y}^{\tiny{\mbox{exp}}}_1} \right)\!.
\end{eqnarray}
The remaining of the proof follows as in Theorem \ref{theo1}.
\end{proof}
Let us observe that the conditions used to obtain (\ref{no2z}) to (\ref{uo2z}) can be summarized as follows:
\begin{equation}\label{ineqsys}
\left\{
\begin{array}{lc}
\frac{1}{p}+\frac{1}{q}\leq 1, \quad  \frac{1}{2}>\frac{N}{2p}, \quad \frac{N}{2}(\frac{1}{p}+\frac{1}{q})>\frac{1}{2}, & \text{for (\ref{no2z})}, \\[0.4cm]
\frac{1}{q}+\frac{1}{r}\leq 1, \quad \frac{1}{2}-\frac{N}{2r}>0, \quad \frac{N}{2}(\frac{1}{q}+\frac{1}{r}-\frac{1}{N})>0, & \text{for (\ref{no3z})},\\[0.4cm]
1-\frac{N}{2q}>0, \quad \frac{1}{2}-\frac{N}{2p}>0,  & \text{for (\ref{co1az})},\\[0.4cm]
\frac{1}{2}-\frac{N}{2p}>0, \quad  \frac{1}{2}-\frac{N}{2}(\frac{1}{q}-\frac{1}{r})>0, \quad
q \leq r,  & \text{for (\ref{co2z})},\\[0.4cm]
\frac{1}{N}+\frac{1}{q}\leq 1, \quad \frac{1}{2}-\frac{N}{2p}>0, \quad 1-\frac{N}{2}(\frac{1}{N}+\frac{1}{q}-\frac{1}{p})>0, &\text{for (\ref{uo2z})}.
\end{array}
\right.
\end{equation}
Thus, the exponents $p,$ $q$ and $r$ must satisfy one of the conditions established in Theorem \ref{theo2}.
\begin{remark}\label{rem2}
Conditions (\ref{ineqsys}) determine the corresponding ones in Theorem \ref{theo2}.
\begin{enumerate}
\item The condition $q \leq r$ is necessary to apply Lemma \ref{heat_kernel2z} and then obtain (\ref{co2z}). This condition together with $\frac{N}{2}(\frac{1}{q}+\frac{1}{r}-\frac{1}{N})>0$ imply that $q<2N$.
\item For (\ref{co1az}), we need $q>\frac{N}{2}$. Then, the range for $q$ is contained in $\frac{N}{2}<q<2N$ that allows the term $N-q$ to be negative, positive or zero. This leads us to the three alternatives in Theorem \ref{theo2}.
\item Notice that if $N = 2,$ the condition $\frac{1}{N}+\frac{1}{q}\leq 1$ implies that $q \geq 2.$ Thus, $q,$ $p$ and $r$ must be taken verifying $\textit{(}ii\textit{)}$ and $\textit{(}iii\textit{)}.$
\end{enumerate}
\end{remark}

\

\textbf{Acknowledgments:} LCFF was partially supported by CNPq and FAPESP, Brazil. The third author has been supported by Fondo Nacional de Financiamiento para la Ciencia, la Tecnología y la Innovación Francisco José de Caldas, contrato Colciencias FP 44842-157-2016.

\

\bibliographystyle{abbrv}

\bibliography{bibchemo_new}

\end{document}